%% file: main.tex
\documentclass[12pt]{scrartcl}

\input{header}

\title{The solvability of the inverse volcano problem over non-prime finite fields}

\author{Alexandru Ghitza and Dhruv Gupta and Maximilian Kortge}

\date{}

\begin{document}

\maketitle

\section{Introduction}


Let $\ell$ be a prime number.
An $\ell$-volcano is a simple connected graph that resembles a volcano (cf.\ \cref{volcano-def}); it consists of some vertices arranged in a crater formation, together with isomorphic trees (lava flows) descending from each crater vertex and stopping at a finite depth $d \geq 0$.
This type of graph occurs naturally in arithmetic geometry (and computational and cryptographic applications):
Kohel \cite{kohel} found that the connected components of the graph $\gpk$ of $\ell$-isogenies between ordinary elliptic curves over a finite field $\fpk$ are $\ell$-volcanoes.
We say that such a volcano \define{appears in} $\gpk$.

\begin{figure}[h!]
  \begin{minipage}{0.48\textwidth}
    \centering
    \begin{tikzpicture}[scale=1]
      \def\rad{1.0}       
      \def\levelsep{1.0}  
      \def\branchsep{0.8} 

      \node[vertex] (t1) at (-0.5,3) {};
      \node[vertex] (t2) at (0.5,3) {};
      \draw[edge] (t1) to[bend right=50] (t2);
      \draw[edge] (t1) to[bend left=50] (t2);

      \node[vertex] (t3) at (-1, 1.5) {};

      \node[vertex] (t4) at (1, 1.5) {};

      \draw[edge] (t1) to[bend left=10] (t3);
      \draw[edge] (t2) to[bend right=10] (t4);
    \end{tikzpicture}
  \end{minipage}
  \hfill
  \begin{minipage}{0.48\textwidth}
    \centering
    \begin{tikzpicture}[scale=1]
      \def\rad{1.0}       
      \def\levelsep{1.0}  
      \def\branchsep{0.8} 

      \node[vertex] (t1) at (-0.5,3) {};
      \node[vertex] (t2) at (0.5,3) {};
      \draw[edge] (t1) to[bend right=50] (t2);
      \draw[edge] (t1) to[bend left=50] (t2);

      \node[vertex] (t3) at (-2, 1.5) {};
      \node[vertex] (t4) at (-1, 1.5) {};

      \node[vertex] (t5) at (1, 1.5) {};
      \node[vertex] (t6) at (2, 1.5) {};

      \node[vertex] (t7) at (-3, 0) {};
      \node[vertex] (t8) at (-2.5, 0) {};
      \node[vertex] (t9) at (-2, 0) {};

      \node[vertex] (t10) at (-1.5, 0) {};
      \node[vertex] (t11) at (-1, 0) {};
      \node[vertex] (t12) at (-0.5, 0) {};

      \node[vertex] (t13) at (0.5, 0) {};
      \node[vertex] (t14) at (1, 0) {};
      \node[vertex] (t15) at (1.5, 0) {};

      \node[vertex] (t16) at (2, 0) {};
      \node[vertex] (t17) at (2.5, 0) {};
      \node[vertex] (t18) at (3, 0) {};

      \node (t19) at (0, -0.25) {};
      \node (t20) at (0, -1) {};

      \draw[edge] (t1) to[bend left=20] (t3);
      \draw[edge] (t1) to[bend left=10] (t4);
      \draw[edge] (t2) to[bend right=10] (t5);
      \draw[edge] (t2)to[bend right=20](t6);

      \draw[edge] (t3) to[bend right=20] (t7);
      \draw[edge] (t3) to[bend right=20] (t8);
      \draw[edge] (t3) to[bend right=10] (t9);
      \draw[edge] (t4) to[bend right=20] (t10);
      \draw[edge] (t4) to[bend right=15] (t11);
      \draw[edge] (t4) to[bend right=10] (t12);

      \draw[edge] (t5) to[bend left=10] (t13);
      \draw[edge] (t5) to[bend left=15] (t14);
      \draw[edge] (t5) to[bend left=20] (t15);
      \draw[edge] (t6) to[bend left=10] (t16);
      \draw[edge] (t6) to[bend left=20] (t17);
      \draw[edge] (t6) to[bend left=20] (t18);
    \end{tikzpicture}
  \end{minipage}
  \vspace{-25pt}
  \caption{A $2$-volcano of depth $1$ (left) and a $3$-volcano of depth $2$ (right).}
  \label{fig:volcano-examples}
\end{figure}
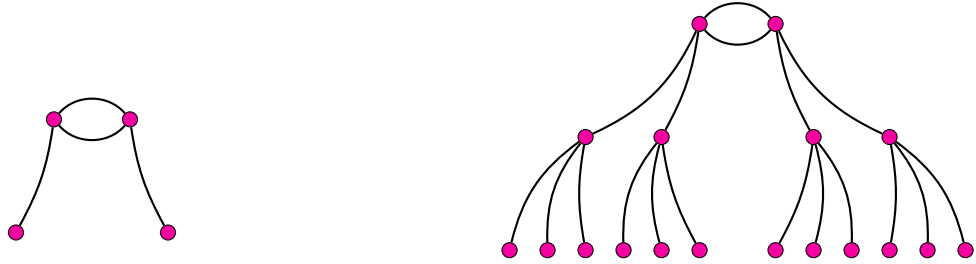

\begin{definition}\label{explosive}
  Consider an $\ell$-volcano $V$ and $k\geq 1$. We say that a prime number $p \ne \ell$ is a \define{\explosive} prime if $V$ appears in $\gpk$.
\end{definition}

Given a fixed $\ell$-volcano $V$ of depth $d \geq 0$ and $k\geq 1$, the inverse volcano problem investigates the existence of such \explosive\ primes.

This problem was first considered by Bambury--Campagna--Pazuki \cite[Question 1.1]{BCP} in the case of prime finite fields $\fp$, where they give a complete positive answer \cite[Theorem 1.3]{BCP}:
any volcano $V$ has infinitely many $1$-explosive primes.

The situation is more delicate if we are looking for a particular volcano in $\gpk$ for $k>1$.
A specific counterexample is given in \cite[Section 5]{BCP}:

\begin{prop}\label{prop:counterexample}
  Let $V$ denote the $2$-volcano illustrated on the left in \cref{fig:volcano-examples}.
  Then $V$ does not appear in $\mathcal{G}_2\left(\FF_{p^2}\right)$ for any prime $p\neq 2$.
  In other words, there are no $2$-explosive primes for $V$.
\end{prop}


Our paper generalises the results of \cite{BCP} by providing a precise framework for investigating the inverse volcano problem for $\gpk$ with a fixed $k>1$.
This is based on a characterisation of the appearance of an $\ell$-volcano in terms of the existence of elements of certain orders in the class groups of appropriate imaginary quadratic fields.

For volcanoes of depth zero, we obtain in \cref{cor:existence_depth_0} a complete positive answer that generalises the corresponding result \cite[Theorem 1.2]{BCP}.
For depth $d\geq 1$, the exact formulation and extent of our results vary with the crater type $V_0$ of the $\ell$-volcano, which in turn depends on the ramification behaviour ($S$plit, $I$nert, or $R$amified) of $\ell$ in the appropriate imaginary quadratic fields (see \cref{craters_orders} for the relation between crater type and ramification).
Letting $r$ denote the $\ell$-adic valuation of $k$, we summarise our findings in the following table:

\begin{center}
  \begin{tabular}{lllllll}
    \toprule
    $V_0$      & $\ell$   & $d$     & $n$                     & Condition on $r$ & \explosive\ primes                 & Reference                                     \\
    \midrule
    $S_n, I_1$ & $\ge 3$  & $\ge 1$ & $\geq 1$                & $r < d$          & Infinitely many                    & \ref{thm:exist_primes-split-inert-general}    \\
               & $2$      & $\ge 2$ & $\geq 1$                & $r < d-1$        & Infinitely many                    & \ref{thm:exist_primes-split-inert-special}    \\
               & $2$      & $1$     & $\geq 1$                & $r = 0$          & Infinitely many                    & \ref{thm:exist_primes-split-inert-special}    \\
    $R_2$      & $\ge 2$  & $\ge 1$ & --                      & $r < d+1$        & Infinitely many                    & \ref{thm:exist_primes-ramified-non-principal} \\
    $R_1$      & $\neq 3$ & $\ge 1$ & --                      & $r < d+1$        & Infinitely many                    & \ref{thm:exist_primes-ramified-principal}     \\
               & $3$      & $\ge 1$ & --                      & $r < d$          & Infinitely many                    & \ref{thm:exist_primes-ramified-principal}     \\

    \midrule

    $S_n, I_1$ & $2$      & $1$     & $\geq 1$                & $r \ge 1$        & None                               & \ref{thm:split_two_converse_low_depth}        \\
               & $2$      & $2,3$   & $\geq 1$                & $r \ge d-1$      & None                               & \ref{thm:split_two_converse_low_depth}        \\
    $S_n$      & $2$      & $\ge 4$ & $\in\{1,2,3,5,6,7\}$    & $r \ge d-1$      & None                               & \ref{thm:split_two_converse_high_depth}       \\
    $R_1$      & $\neq 3$ & $\ge 1$ & --                      & $r \ge d+1$      & None                               & \cref{thm:converse_R1}                        \\
               & $3$      & $\ge 1$ & --                      & $r \ge d$        & None                               & \cref{thm:converse_R1}                        \\
    \midrule
    $I_1$      & $\ge 3$  & $\ge 1$ & --                      & $r \ge d$        & Infinitely many\textsuperscript{*} & \ref{heuristics-inert}                        \\
    $R_2$      & $\ge 5$  & $\ge 1$ & --                      & $r \ge d+1$      & Infinitely many\textsuperscript{*} & \ref{heuristics-ramified}                     \\
    \midrule
    $S_n$      & $\ge 3$  & $\ge 1$ & $\geq 1$                & $r \ge d$        & ??                                 & --                                            \\
    $S_n$      & $2$      & $\ge 4$ & $\notin\{1,2,3,5,6,7\}$ & $r \ge d-1$      & ??                                 & --                                            \\
    $I_1$      & $2$      & $\ge 4$ & --                      & $r \ge d-1$      & ??                                 & --                                            \\
    $R_2$      & $2,3$    & $\ge 1$ & --                      & $r \ge d+1$      & ??                                 & --                                            \\
    \bottomrule
    \multicolumn{7}{l}{\footnotesize *\,\,\, Result conditional on modified \cohlen\ heuristics.}                                                                     \\
    \multicolumn{7}{l}{\footnotesize ?? Our methods yield no information.}                                                                                            \\
  \end{tabular}
\end{center}

In some of the cases (denoted with a {*}) our method requires detailed knowledge of the behaviour of the $\ell$-torsion in class groups of imaginary quadratic fields.
This is a difficult problem that remains an active area of research, so we had to resort to information that is conditional on a modification of the \cohlen\ heuristics.
There are also some cases (denoted with ?? at the bottom of the table) where we do not obtain any results, conditional or not.

In the process of proving the appearance of a volcano, we need to find a suitable order $\oh_0$ in an imaginary quadratic field $K$, corresponding to the crater of the volcano.
In almost all cases, we find that $\oh_0$ can be chosen to be the maximal order $\oh_K$, cf. \cref{rem:maximal_orders_for_existence} and \cref{subsect:cohen-lenstra}.
(The exception is any $3$-volcano with crater $R_1$, where $\oh_0$ is the order of conductor $2$ in $\oh_K$.)

Here is an outline of the paper.
In \cref{sect:prelim} we gather some notation and basic facts about the fundamental objects we work with: orders in imaginary quadratic fields, elliptic curves and isogenies, volcano graphs.
In \cref{sect:orders} we focus on the class groups of imaginary quadratic orders, especially how they vary as we increase the power of $\ell$ dividing the conductor of the order.
\cref{sect:solvability} contains our general criterion for the existence/non-existence of \explosive\ primes for a fixed volcano.
(An interesting feature, stemming from the Chebotarev Density Theorem and noticeable in the above table, is that the set of \explosive\ primes either has positive density or is empty.)
In \cref{sect:appearance} we apply this criterion to solve the inverse problem, first for craters (based on results from \cite{BCP}) and then for deeper volcanoes with various crater types.
Our modification of the \cohlen\ heuristics, a description of the computational evidence supporting them, and the conditional results they give rise to are in \cref{sect:computations}.
The appendices gather some supporting material that does not appear in the literature in the form we need it.

\textbf{Acknowledgements:}
The point of origin of this research was Fabien Pazuki's Behrend Memorial Lecture at the University of Melbourne in April 2024.
We express our gratitude to Pazuki for introducing us to the inverse volcano problem, and to the Behrend Fund for making the event possible.
The work of the authors was funded in part by the Vacation Scholarship Program and a grant from the School of Mathematics and Statistics at the University of Melbourne, while the computational component was supported by the University's Research Computing Services and the Petascale Campus Initiative.

\section{Preliminaries}\label{sect:prelim}

Fix a prime $\ell$.

\subsection{Imaginary quadratic orders}\label{subsect:prelim_orders}

Let $K=\QQ\big(\sqrt{N}\big)$, $N<0$ squarefree integer, be an imaginary quadratic field with ring of integers $\oh_K$.
We denote by $D_K$ the discriminant of $\oh_K$:
\begin{equation*}
  D_K =
  \begin{cases}
    N  & \text{if }N\equiv 1\mod{4}, \\
    4N & \text{otherwise.}
  \end{cases}
\end{equation*}
We take the $\ZZ$-basis $\{1,\omega_K\}$ of $\oh_K$, where
\begin{equation*}
  \omega_K = \frac{D_K+\sqrt{D_K}}{2}.
\end{equation*}

Let $c\geq 1$ be coprime to $\ell$ and consider the order $\oh_0$ of conductor $c$: this is a subring of $\oh_K$ of index $c$ and
\begin{equation*}
  \oh_0 = \ZZ + c\oh_K = \ZZ + c\omega_K\ZZ.
\end{equation*}
For $d\geq 0$, let $\oh_d$ be the order of conductor $c\ell^d$.

The \define{ideal class group} $\Cl(\oh_d)$ (sometimes called the \define{Picard group} of $\oh_d$) is the group of proper fractional $\oh_d$-ideals modulo principal $\oh_d$-ideals (cf. \cite[\S 7]{cox}).
It is a finite abelian group.

The \define{ring class field} of $\oh_d$ is the unique finite abelian extension $H_d$ of $K$ such that
\begin{itemize}
  \item all the prime ideals of $K$ that ramify in $H_d$ divide $c\ell^d\oh_K$;
  \item the ideal class group $\Cl(\oh_d)$ is canonically isomorphic to the Galois group $\Gal(H_d/K)$ under the Artin reciprocity map.
\end{itemize}
For $d'\geq d$ we have $\oh_{d'}\subseteq\oh_d$ and $H_d\subseteq H_{d'}$.
(See \cite[Exercise 9.19]{cox}.)

For more details about ring class fields, see \cite[Chapter 9]{cox}.

\subsection{Some discriminant upper bounds}
We continue using the notation from the previous section.
For any ideal $\LL$ of $\oh_0$, we write $\ord(\LL)$ to denote the order of $\LL$ in the class group $\Cl(\oh_0)$.

\begin{lem}\label{noninteg_norm}
  If $\alpha\in\oh_0\setminus\ZZ$, then $N(\alpha)\geq |D_0|/4$.
\end{lem}
\begin{proof}
  Write $\alpha=a+bc_0 \omega_K$ with $a,b \in \ZZ$.
  We have
  \begin{align*}
    4 N(\alpha) & = (2 \alpha)(2 \overline{\alpha})                                                                     \\
                & = \left((2 a + b c_0 D_K) + b c_0 \sqrt{D_K}\right) \left((2 a + b c_0 D_K) - b c_0 \sqrt{D_K}\right) \\
                & = (2 a + b c_0 D_K)^2 - b^2 c_0^2 D_K                                                                 \\
                & = \big(\Tr(\alpha)\big)^2 + b^2 |D_0|.
  \end{align*}
  Since $\alpha\notin \ZZ$ we have $b \neq 0$ so that $N(\alpha) \geq |D_0| / 4$.
\end{proof}
\begin{rem}
  The special case $\oh_0=\oh_K$ of \cref{noninteg_norm} is \cite[Lemma 1]{boyd}.
\end{rem}

\begin{prop}\label{disc_upper_bound}
  Consider the decomposition of $\ell\oh_0$ into prime ideals of $\oh_0$.
  \begin{enumerate}[label=(\alph*), ref=Proposition~\theprop(\alph*)]
    \item\label{disc_upper_bound-split} If $\ell\oh_0=\LL\widebar{\LL}$, with $\LL\neq\widebar{\LL}$ prime ideals and $\ord(\LL)=n$, then $|D_0|\leq 4\ell^n-1$.
      (This is \cite[Lemma 4.9]{BCP}.)
    \item\label{disc_upper_bound-ramified} If $\ell\oh_0=\LL^2$, with $\LL$ prime and principal, then $|D_0|\leq 4\ell$.
      More precisely, $D_0=c_0^2D_K$ with
      \begin{equation*}
        (D_K,c_0)\in\begin{cases}
          \big\{(-4\ell,1)\big\}          & \text{if }\ell\equiv 1\mod{4}, \\
          \big\{(-\ell,1),(-\ell,2)\big\} & \text{if }\ell\equiv 3\mod{4}, \\
          \big\{(-8,1),(-4,1)\big\}       & \text{if }\ell=2.
        \end{cases}
      \end{equation*}
  \end{enumerate}
\end{prop}
\begin{proof}
  \
  \begin{enumerate}
    \item
      Write $\LL^n=\alpha\oh_0$ for some $\alpha\in\oh_0$.
      We have $N(\alpha)=N(\LL)^n=\ell^n$.

      If $\alpha\in\ZZ$ then $\alpha^2=\alpha\overline{\alpha}=N(\alpha)=\ell^n$ implies that $n$ is even and $\alpha=\pm\ell^{n/2}$, therefore $\LLbar=\LL$, contradiction.

      So $\alpha\in\oh_0\setminus\ZZ$, therefore by \cref{noninteg_norm} we get that $|D_0|\leq 4\ell^n$.

      But if equality holds then $c_0^2|D_K|=|D_0|=4\ell^n$.
      Since the conductor $c_0$ of $\oh_0$ is assumed to be coprime to $\ell$, this forces $\ell$ to divide $D_K$, contradicting the fact that $\ell$ is split in $K$.
    \item
      Writing $\LL=\alpha\oh_0$, the reasoning from part (a) tells us that $\alpha\in\oh_0\setminus\ZZ$, so that $|D_0|\leq 4\ell$.
      Now $D_0=c_0^2 D_K$ with $c_0$ coprime to $\ell$, and $D_K$ negative, divisible by $\ell$, and a fundamental discriminant (either $1\mod{4}$ and squarefree, or $4$ times a squarefree integer that is $2,3\mod{4}$).
      This gives the following options for $D_K$:
      \begin{equation*}
        D_K\in\begin{cases}
          \{-4\ell,-3\ell\} & \text{if }\ell\equiv 1\mod{4}, \\
          \{-\ell\}         & \text{if }\ell\equiv 3\mod{4}, \\
          \{-8,-4\}         & \text{if }\ell=2.
        \end{cases}
      \end{equation*}
      We can rule out the case $\ell\equiv 1\mod{4}$ and $D_K=-3\ell$ as follows.
      Here $\oh_K=\ZZ[\sqrt{-3\ell}]$, and we are assuming that $\ell\oh_K=(\alpha\oh_K)^2$ for some $\alpha\in\oh_K$.
      Writing $\alpha=a+b\sqrt{-3\ell}$ with $a,b\in\ZZ$, we have
      \begin{equation*}
        \pm\ell = \alpha^2 = a^2+2ab\sqrt{-3\ell}-3b^2\ell,
      \end{equation*}
      impossible.

      The leaves us with the cases in the statement, and it is easy to see that $c_0=1$ for all of them except when $\ell\equiv 3\mod{4}$.
      \qedhere
  \end{enumerate}
\end{proof}

\subsection{Elliptic curves and isogenies}

Fix a prime $p$ distinct from $\ell$, and $k\geq 1$.
We consider elliptic curves over $\fpk$ and isogenies between them.

The endomorphism ring $\oh$ of an ordinary elliptic curve $E$ over $\fpk$ is an imaginary quadratic order.
(We say that $E$ has \define{complex multiplication}, \define{CM} for short, by $\oh$.)
The ring $\oh$ encodes a good deal of arithmetic information about $E$, such as
\begin{prop}
  Let $\varphi\st E_1 \to E_2$ be an $\ell$-isogeny of (ordinary) elliptic curves with CM by orders $\oh_1$ and $\oh_2$, respectively.
  Then one of the following is true:
  \begin{enumerate}
    \item $\oh_1 = \oh_2$ and we say that $\varphi$ is \define{horizontal};
    \item $[\oh_1 : \oh_2] = \ell$ and we say that $\varphi$ is \define{descending};
    \item $[\oh_2 : \oh_1] = \ell$ and we say that $\varphi$ is \define{ascending}.
  \end{enumerate}
  In the last two cases we say that $\varphi$ is \define{vertical}.
\end{prop}
\begin{proof}
  See \cite[Proposition 21]{kohel} or \cite[\S2.7]{sutherland}.
\end{proof}

\subsection{Volcano graphs}

\begin{definition}\label{volcano_craters}
  A \define{(volcano) crater} $V_0$ is a regular connected undirected graph that is one of:

  \begin{center}
    \renewcommand{\arraystretch}{1.5}
    \begin{tabular}{clc}
      \toprule
      Notation & description                             & picture                                \\
      \midrule
      $I_1$    & a single vertex                         &
      \begin{tikzpicture}[scale=1, every node/.style={vertex}, baseline={([yshift=-.5ex]current bounding box.center)}]
        \node at (0,0) {};
      \end{tikzpicture}              \\
      $R_1$    & a single vertex with a single self-loop &
      \begin{tikzpicture}[scale=1, every node/.style={vertex}, baseline={([yshift=-.5ex]current bounding box.center)}]
        \node (a1) at (0,0) {};
        \draw[edge] (a1) .. controls +(0.6,-0.7) and +(0.6,0.7) .. (a1);
      \end{tikzpicture}              \\
      $S_1$    & a single vertex with two self-loops     &
      \begin{tikzpicture}[scale=1, every node/.style={vertex}, baseline={([yshift=-.5ex]current bounding box.center)}]
        \node (b1) at (0,0) {};
        \draw[edge] (b1) .. controls +(0.6,-0.7) and +(0.6,0.7) .. (b1);
        \draw[edge] (b1) .. controls +(-0.6,-0.7) and +(-0.6,0.7) .. (b1);
      \end{tikzpicture}              \\
      $R_2$    & two vertices connected by a single edge &
      \begin{tikzpicture}[scale=1, every node/.style={vertex}, baseline={([yshift=-.5ex]current bounding box.center)}]
        \node (d1) at (0,0) {};
        \node (d2) at (1,0) {};
        \draw[edge] (d1) to (d2);
      \end{tikzpicture}              \\[+10pt]
      $S_2$    & two vertices connected by two edges     &
      \begin{tikzpicture}[scale=1, every node/.style={vertex}, baseline={([yshift=-.5ex]current bounding box.center)}]
        \node (e1) at (0,0) {};
        \node (e2) at (1,0) {};
        \draw[edge] (e1) to[bend right=50] (e2);
        \draw[edge] (e1) to[bend left=50] (e2);
      \end{tikzpicture}              \\[+10pt]
      $S_n$    & $n\geq 3$ vertices arranged in a cycle  &
      \begin{tikzpicture}[scale=1, every node/.style={vertex}, baseline={([yshift=-.5ex]current bounding box.center)}]
        \def\n{5} 
        \foreach \i in {1,...,\n}{
            \node (h\i) at ({-0.5 - cos(360/\n*\i)}, {sin(360/\n*\i)}) {};
          }
        \foreach \i [evaluate=\i as \j using {int(mod(\i,\n)+1)}] in {1,...,\n}{
            \ifnum\i=2
              \draw[edge,dashed] (h\i)--(h\j); 
            \else
              \draw[edge] (h\i)--(h\j);        
            \fi
          }
      \end{tikzpicture} \\
      \bottomrule
    \end{tabular}
  \end{center}

\end{definition}

\begin{definition}\label{volcano-def}
  An \define{$\ell$-volcano graph} of depth $d \geq 0$, denoted $(V_0, \ell, d)$, is a connected undirected graph whose vertices admit a partition into levels $\bigsqcup_{i=0}^d \mathcal{V}_i$, where $V_i$ denotes the subgraph induced by $\mathcal{V}_i$, such that:
  \begin{enumerate}
    \item The subgraph $V_0$ is a volcano crater as described in \cref{volcano_craters}.
    \item For all $0 < i \leq d$, the subgraph $V_i$ is totally disconnected.
    \item For $0 \leq i < d$, each vertex in $\mathcal{V}_i$ has degree $\ell + 1$; and each vertex in $\mathcal{V}_d$ has degree $1$.
    \item For $0 < i \leq d$, each vertex in $\mathcal{V}_i$ has exactly one neighbour in $\mathcal{V}_{i-1}$.
  \end{enumerate}
\end{definition}

\subsection{Volcanoes inside isogeny graphs}

Let $\Phi_\ell\in\FF_p[X,Y]$ denote the modulo $p$ reduction of the $\ell$-modular polynomial (see \cite[Section 2.3]{sutherland}).
It has the property that for all $j_1,j_2\in\fpk$, we have $\Phi_\ell(j_1,j_2)=0$ if and only if $j_1,j_2$ are the $j$-invariants of elliptic curves over $\fpk$ that are related by an $\ell$-isogeny defined over $\fpk$.

Let $S\subseteq\fpk$ denote the set of $j$-invariants of supersingular elliptic curves over $\fpk$.
In this paper, the (ordinary, simple) \define{$\ell$-isogeny graph $\gpk$} has as vertex set $\fpk\setminus (S\cup \{0,1728\})$, and there is an edge $(j_1,j_2)$ with the same multiplicity as that of $j_2$ as a root of $\Phi_\ell(j_1,Y)$.
Due to our restrictions on the vertex set, the graph $\gpk$ can be considered to be undirected.

Exploiting the dictionary between elliptic curves and orders, Kohel determined the structure of the connected components of the isogeny graph $\gpk$:
\begin{thm}[Kohel]
  Let $V$ be a connected component of $\gpk$, with vertex set $\mathcal{V}$.
  Fix some elliptic curve $E/\FF_{p^k}$ such that $j(E) \in \mathcal{V}$ and let $K=\End(E)\otimes\QQ$.
  Let $\pi_E$ denote the Frobenius endomorphism of $E$, let $v=\left[\oh_K : \ZZ\left(\pi_E\right)\right]$ and $d=\nu_\ell(v)$.

  Then $V$ is an $\ell$-volcano of depth $d$ such that
  \begin{enumerate}
    \item for $0\leq i\leq d$, every $E/\FF_{p^k}$ with $j(E) \in \mathcal{V}_i$ has CM by $\oh_i$, an order in $K$;
    \item $\ell \nmid \left[\oh_K : \oh_0\right]$ and for $0 \leq i < d$, $\left[\oh_i : \oh_{i+1}\right] = \ell$;
    \item the crater $V_0$ is regular of degree $1 + \left(\frac{D_K}{\ell}\right)$ and has $\ord_{\Cl(\oh_0)}\left[\LL\right]$ vertices, where $\LL$ is a prime ideal dividing $\ell\oh_0$.
  \end{enumerate}
\end{thm}

Moreover, as explained in \cite[Proposition 3.14 and Definition 4.2]{BCP}, we can place the crater $V_0$ in the classification from \cref{volcano_craters} as follows:
\begin{prop}\label{craters_orders}
  Let $V_0$ be the crater of a volcano in $\gpk$, let $\oh_0$ be the endomorphism ring of any vertex in $V_0$, and let $\LL \subseteq \oh_0$ be a prime ideal above $\ell$.
  Then, $V_0$ is of type
  \begin{enumerate}
    \item $I_1$ if $\ell$ is inert in $\oh_0$;
    \item $R_1$ if $\ell$ is ramified in $\oh_0$ and $\LL$ is principal;
    \item $S_1$ if $\ell$ is split in $\oh_0$ and $\LL$ is principal;
    \item $R_2$ if $\ell$ is ramified in $\oh_0$ and $\LL$ is not principal;
    \item $S_n$ if $\ell$ is split in $\oh_0$ and $[\LL]$ has order $n\geq 2$ in $\Cl(\oh_0)$.
  \end{enumerate}

\end{prop}

\begin{rem}
  This Proposition explains our notation for volcano craters in \cref{volcano_craters}: the letter stands for the decomposition type (inert, ramified, split), and the subscript represents the number of vertices in the crater, which is at the same time the order of a prime ideal above $\ell$ in the class group.
\end{rem}

\begin{definition}\label{compatible_order}
  Given a crater $V_0$ and a prime $\ell$, we say that an order $\oh_0$ in an imaginary quadratic field $K$ is \define{compatible} with $(V_0,\ell)$ if the conductor of $\oh_0$ is coprime to $\ell$ and $\oh_0$ corresponds to the type of $V_0$ as described in \cref{craters_orders}.
\end{definition}

\section{Class groups of imaginary quadratic orders}\label{sect:orders}

Recall the notation from \cref{subsect:prelim_orders}: $\oh_0$ is an order of conductor $c$ (coprime to $\ell$) in an imaginary quadratic field $K$.
For each $d\geq 0$, we let $\oh_d$ be the order of conductor $c\ell^d$.

If $d'\geq d$, then the natural surjective homomorphism between the class groups is
\begin{equation*}
  \Cl(\oh_{d'})\to \Cl(\oh_d)\qquad\text{given by }[\pp\cap \oh_{d'}]\mapsto [\pp\cap \oh_d]\text{ for any prime ideal }\pp\subseteq \oh_K.
\end{equation*}
Let $S_d=\Cl(\oh_d)[\ell^\infty]$ denote the Sylow $\ell$-subgroup of the class group.
We let $\kappa_d$ be the kernel of the restriction of the natural homomorphism $\Cl(\oh_d)\to\Cl(\oh_0)$ to the Sylow $\ell$-subgroups:
\begin{equation}\label{eq:ses}
  1\to \kappa_d\to S_d\xrightarrow{\;\;\varphi_d\;\;} S_0\to 1.
\end{equation}

\subsection{The kernels}

To describe the kernels $\kappa_d$ we use another short exact sequence:
\begin{equation}\label{eq:kd}
  1\to \big(\ZZ/\ell^d\ZZ\big)^\times[\ell^\infty] \xrightarrow{\;\;\iota\;\;} \big(\oh_K/\ell^d\oh_K\big)^\times[\ell^\infty] \to \kappa_d\to 1.
\end{equation}
Sequences \eqref{eq:ses} and \eqref{eq:kd} are special cases of the much more general \cite[Propositions 6.4 and 6.6]{kopp}.
We include in the appendix (see \cref{kappa_n}) a simple argument that deduces them from the special case $\oh_0=\oh_K$ that appears in \cite[Equations (7.25) and (7.27)]{cox}.

We also need to relate $\kappa_{d+1}$ and $\kappa_d$ in a precise manner.
Consider the commutative diagram given below, where the map $\pi'_d\st\kappa_{d+1}\to \kappa_d$ is the restriction of the natural surjective homomorphism $\pi_d\st S_{d+1}\to S_d$.
It is easy to see that
\begin{equation*}
  \ker \big(\pi_d\st S_{d+1}\to S_d\big)\subseteq \kappa_{d+1},
\end{equation*}
so we can make the identifications in the diagram:
\begin{equation}\label{kappa_diag}
  \begin{tikzcd}[column sep=2em,row sep=1.5em,trim left=-4cm]
    & 1 \arrow[d] & 1 \arrow[d] \\
    & \ker \pi'_d \arrow[d] \arrow[r, equal] & \ker \pi_d \arrow[d]\\
    1 \arrow[r] & \kappa_{d+1} \arrow[r] \arrow[d, "\pi'_d"] & S_{d+1} \arrow[r, "\varphi_{d+1}"] \arrow[d, "\pi_d"] & S_0 \arrow[r] \arrow[d, equal] & 1 \\
    1 \arrow[r] & \kappa_d \arrow[r] \arrow[d] & S_d \arrow[r, "\varphi_d"] \arrow[d] & S_0 \arrow[r] & 1 \\
    & 1 & 1
  \end{tikzcd}
\end{equation}

The following result completely describes the groups $\kappa_d$ and the maps $\pi'_d\st\kappa_{d+1}\to\kappa_d$.
It is obtained by restricting to Sylow $\ell$-subgroups in \cref{lambda-classification}, which distills the results on the unit groups of quadratic orders from \cite{sittinger} and adapts them to our needs.

\begin{thm}\label{kappa-classification}
  For $d=0$ we have $\kappa_0=1$ and $\pi'_0\st\kappa_1\to\kappa_0=1$.

  Assume now that $d \geq 1$.
  In each of the following cases, the given conditions on $d$, $\ell$, and $\oh_K$ imply the stated group structure for $\kappa_d$, and the stated type of morphism $\pi'_d\st\kappa_{d+1}\to\kappa_d$.
  (Below, ``$\can$'' refers to the canonical surjections of type $\ZZ/\ell^{n+1}\ZZ\to\ZZ/\ell^n\ZZ$.)

  \begin{center}
    \settabprefix{Theorem~\thethm}
    \setcounter{tabenum}{0}
    \begin{tabular}{lllllll}
      \toprule
                                                & $d$      & $\ell$   & $\ell$ in $\oh_K$ & extra assumptions      & $\kappa_d$                           & $\pi'_d$          \\ \midrule
      \tabitem{kappa-class-split-inert-general} & $\geq 1$ & $\geq 3$ & split or inert    & none                   & $\ZZ/\ell^{d-1}\ZZ$                  & $\can$            \\ \midrule
      \tabitem{kappa-class-ramified-general}    & $\geq 1$ & $=2$     & ramified          & $D_K\equiv 8\mod{16}$  & $\ZZ/\ell^d\ZZ$                      & $\can$            \\
                                                &          & $=3$     &                   & $D_K\equiv 3\mod{9}$                                                              \\
                                                &          & $\geq 5$ &                   & none                                                                              \\ \midrule
      \tabitem{kappa-class-ramified-special}    & $\geq 1$ & $=2$     & ramified          & $D_K\equiv 12\mod{16}$ & $\ZZ/\ell^{d-1}\ZZ\times\ZZ/\ell\ZZ$ & $\can \times \id$ \\
                                                &          & $=3$     &                   & $D_K\equiv 6\mod{9}$                                                              \\ \midrule
      \tabitem{kappa-class-split-inert-special} & $\geq 2$ & $=2$     & split or inert    & none                   & $\ZZ/2^{d-2}\ZZ\times\ZZ/2\ZZ$       & $\can \times \id$ \\
      \bottomrule
    \end{tabular}
  \end{center}

  The remaining special case (in (d)) is: for $\ell=2$ split or inert and $d<2$ we have $\kappa_1\cong 0$.
\end{thm}
\begin{proof}
  Follows directly from \cref{lambda-classification} by taking Sylow $\ell$-subgroups and observing how the homomorphism $\pi'_d$ maps the explicitly described generators.
\end{proof}

\begin{rem}
  Just as in \cref{lambda-classification}, the congruence classes for $D_K$ listed for $\ell=2,3$ are the only possible ones given that we assume $\ell$ to ramify in $\oh_K$.
\end{rem}

We summarise the information about $\ker\pi'_d$ from \cref{kappa-classification} here:
\begin{cor}\label{kerpid}
  We have
  \begin{equation*}
    \ker\pi'_0\cong
    \begin{cases}
      \ZZ/\ell\ZZ & \text{if $\ell$ is ramified in $\oh_K$} \\
      1           & \text{otherwise}
    \end{cases}
    \quad\text{and}\quad
    \ker\pi'_d\cong\ZZ/\ell\ZZ\qquad\text{for all }d\geq 1.
  \end{equation*}
\end{cor}

\subsection{Split exact sequence}

In the special case where the short exact sequence
\begin{equation}\label{sn_ses}
  1\to \kappa_n\to S_n\xrightarrow{\;\;\varphi_n\;\;} S_0\to 1
\end{equation}
splits, we can give a more precise description of $\kappa_n$ and how it varies with $n$:
\begin{lem}\label{lem:splitting}
  Given an isomorphism $h_{d+1}\st S_{d+1}\to \kappa_{d+1}\times S_0$ coming from a splitting of the short exact sequence \eqref{sn_ses} for $n=d+1$, there exists an isomorphism $h_d\st S_d\to \kappa_d\times S_0$ (associated with a splitting of \eqref{sn_ses} for $n=d$) such that the following square commutes:
  \begin{equation*}
    \begin{tikzcd}[column sep=4em,row sep=2em,trim left=-2.5cm]
      \kappa_{d+1}\times S_0 \arrow[d, "\pi'_d\times\id_{S_0}"]
      & S_{d+1} \arrow[d, "\pi_d"] \arrow[l, "h_{d+1}"]\\
      \kappa_d\times S_0
      & S_d. \arrow[l, "h_d"]
    \end{tikzcd}
  \end{equation*}
  In particular, $\ker\pi_d$ is identified with $\ker \big(\pi'_d\times\id_{S_0}\big)=\big(\ker\pi'_d\big)\times \{1\}$.
\end{lem}
\begin{proof}
  Consider the following diagram:
  \begin{equation*}
    \begin{tikzcd}[row sep=1.5em, column sep=1.5em]
      & \kappa_{d+1} \arrow[dl, equal] \arrow[rr] \arrow[dd]
      & & S_{d+1} \arrow[dl, "h_{d+1}", "\sim"'] \arrow[rr] \arrow[dd]
      & & S_0 \ar[dl, equal] \arrow[dd, equal] \\
      \kappa_{d+1} \arrow[rr, crossing over] \arrow[dd, "\pi'_d", pos=0.4]
      & & \kappa_{d+1}\times S_0 \arrow[rr, crossing over]
      & & S_0 \\
      & \kappa_d \arrow[dl, equal] \arrow[rr]
      & & S_d \arrow[dl, "h_d", "\sim"'] \arrow[rr]
      & & S_0 \ar[dl, equal] \\
      \kappa_d \arrow[rr]
      & & \kappa_d\times S_0 \ar[rr] \arrow[from=uu, crossing over, "\alpha", pos=0.4]
      & & S_0. \arrow[from=uu, equal, crossing over]
    \end{tikzcd}
  \end{equation*}

  It is obtained as follows:
  \begin{itemize}
    \item The two squares at the back are commutative and come from diagram \eqref{kappa_diag}.
    \item The two squares at the top are commutative and come from the splitting of \eqref{sn_ses} for $n=d+1$.
    \item The fact that the splitting of \eqref{sn_ses} for $n=d+1$ implies the splitting for $n=d$ follows from:
      if $\sigma_{d+1}\st S_0\to S_{d+1}$ is a right inverse of $\varphi_{d+1}$, then $\sigma_d\coloneq \pi_d\circ\sigma_{d+1}$ is a right inverse of $\varphi_d$.

      Thus we get the commutativity of the two squares at the bottom.
    \item The commutativity of the leftmost and rightmost side squares is obvious.
    \item We ensure the commutativity of the middle side square by defining the vertical arrow $\alpha\st\kappa_{d+1}\times S_0\to\kappa_d\times S_0$ to be whatever makes the square commute: this can be done because two of the other arrows in the same square are invertible.
  \end{itemize}
  The rest of this proof is dedicated to showing that $\alpha=\pi'_d\times \id_{S_0}$.
  This is a simple consequence of the commutativity of the front two squares, which we obtain from the Cube Lemma~\cite[Proposition 1.1]{mitchell} as follows.
  Consider the cube sitting on the left of the diagram.
  We have seen that all faces except the front one are commutative, and the map in the upper left corner $\id\st\kappa_{d+1}\to\kappa_{d+1}$ is surjective, hence by the Cube Lemma the front face is also commutative.

  We apply the same argument to the cube on the right of the diagram, as all the non-front faces are commutative, and the map in the upper left corner $h_{d+1}\st S_{d+1}\to\kappa_{d+1}\times S_0$ is an isomorphism, hence surjective.
\end{proof}
\begin{rem}
  The map $\pi'_d\st\kappa_{d+1}\to\kappa_d$ induces a map $\Ext^1(S_0,\kappa_{d+1})\to\Ext^1(S_0,\kappa_d)$;
  one can check that this takes the sequence \eqref{sn_ses} for $n=d+1$ to the one for $n=d$.
\end{rem}

\subsection{The case of two-torsion}

Motivated by the previous section, we specialise to the case $\ell=2$ and investigate the splitting of the short exact sequence
\begin{equation}\label{ses_two}
  1\to \kappa_d\to S_d\xrightarrow{\;\;\varphi_d\;\;} S_0\to 1.
\end{equation}

\begin{lem}\label{ses-card-split}
  If $d\geq 0$ is such that the sequence \eqref{ses_two} splits, then
  \begin{equation}\label{ses_card}
    \card{S_d[2]}=\card{S_0[2]}\cdot \card{\kappa_d[2]}.
  \end{equation}
  The converse is true if $\kappa_d$ or $\Cl(\oh_0)$ has no elements of order $4$.
\end{lem}
\begin{proof}
  One direction is obvious: if $S_d\cong S_0\times \kappa_d$ then $S_d[N]\cong S_0[N]\times \kappa_d[N]$ for all $N\geq 1$.

  In the other direction, applying the left-exact functor $\Hom(\ZZ/2\ZZ, -)$ to sequence \eqref{ses_two} gives
  \begin{equation}\label{ses_two_torsion}
    1\to \kappa_d[2]\to S_d[2]\xrightarrow{\;\;\varphi_d'\;\;} S_0[2] \to 1,
  \end{equation}
  where the surjectivity of $\varphi_d'$ came from \cref{ses_card}.

  If $\Cl(\oh_0)$ has no elements of order $4$ then $S_0[2]=S_0$.
  Viewing the groups in \eqref{ses_two_torsion} as $\FF_2$-vector spaces, there exists a right inverse of $\varphi_d'$, which also works as a right inverse of $\varphi_d$.

  Otherwise $\kappa_d$ has no elements of order $4$ and so $\kappa_d[2]=\kappa_d$.
  Consider the natural surjection $S_d\to S_d/2S_d$.
  Let $a\in\kappa_d\cap2S_d$, so $a=2b$ for some $b\in S_d$.
  Then $2\varphi_d(b) = 0$ and $\varphi_d(b)\in S_0[2]$.
  By the surjectivity of $\varphi_d'$ there exists $c\in S_d[2]$ such that $\varphi_d'(c)=\varphi_d'(b)$, hence $b-c\in\kappa_d$ and we have $2(b-c)=2b=a\in2\kappa_d$.
  But $\kappa_d[2]=\kappa_d$, so $a = 0$.
  Therefore, when restricted to $\kappa_d$, the natural surjection is injective.
  Since $S_d/2S_d$ is an $\FF_2$-vector space and $\kappa_d$ maps isomorphically onto a subspace, the composition with the projection $S_d \to S_d/2S_d\to \kappa_d$ gives a left inverse of the inclusion $\kappa_d\to S_d$.
\end{proof}

It is clear that we would benefit from having a good handle on the $2$-torsion in the class groups of imaginary quadratic fields.
Thanks to classical genus theory, this $2$-torsion can be described very precisely, as explained by Cox:
\begin{prop}[{\cite[Proposition 3.11]{cox}}]\label{two-torsion}
  Let $\oh$ be an imaginary quadratic order with discriminant $D$.
  Let $\eta$ be the number of odd prime divisors of $D$.
  Let
  \begin{equation*}
    \tau=\begin{cases}
      \eta-1 & \text{if }D\equiv 1\mod{4},                   \\
      \eta-1 & \text{if }D=4y\text{ and }y\equiv 1\mod{4},   \\
      \eta   & \text{if }D=4y\text{ and }y\equiv 2,3\mod{4}, \\
      \eta   & \text{if }D=4y\text{ and }y\equiv 4\mod{8},   \\
      \eta+1 & \text{if }D=4y\text{ and }y\equiv 0\mod{8}.
    \end{cases}
  \end{equation*}
  Then $\card{S[2]}=\card{\Cl(\oh)}[2]=2^{\tau}$, where $S$ is the Sylow $2$-subgroup of the class group $\Cl(\oh)$.
\end{prop}

\begin{lem}\label{ses-card-holds}
  Let $\ell=2$, let $\oh_0$ be an order of odd conductor $c_0$ in an imaginary quadratic field $K$, and let $d\geq 0$.
  \begin{enumerate}
    \item If $2$ is split or inert in $\oh_0$, then \cref{ses_card} holds for all $d\geq 0$.
    \item If $2$ is ramified in $\oh_0$ and $D_K\equiv 8\mod{16}$, then \cref{ses_card} holds for all $d\geq 0$.
    \item If $2$ is ramified in $\oh_0$ and $D_K\equiv 12\mod{16}$, then \cref{ses_card} holds only for $d=0$.
  \end{enumerate}
\end{lem}
\begin{proof}
  The argument is based on a straightforward but tedious case-by-case analysis building on \cref{two-torsion}.

  Let $D_0$ be the discriminant of $\oh_0$, then the discriminant of $\oh_d$ is $D_d=2^{2d}D_0$, and $D_d$ has the same number $\eta$ of odd prime divisors for all $d\geq 0$.

  \begin{enumerate}
    \item If $2$ is split or inert in $\oh_0$, then $D_0\equiv 1$ or $5\mod{8}$.
      From \cref{two-torsion}, resp.\ \ref{kappa-class-split-inert-special} we get
      \begin{equation*}
        \card{S_d[2]} =
        \begin{cases}
          2^{\eta-1} & \text{if }d=0,1,   \\
          2^{\eta}   & \text{if }d=2,     \\
          2^{\eta+1} & \text{if }d\geq 3,
        \end{cases}
        \qquad\text{resp.}\qquad
        \card{\kappa_d[2]} =
        \begin{cases}
          2^0 & \text{if }d=0,1,   \\
          2^1 & \text{if }d=2,     \\
          2^2 & \text{if }d\geq 3,
        \end{cases}
      \end{equation*}
      giving \cref{ses_card} for all $d\geq 0$.
    \item Similarly, by \cref{two-torsion} and \ref{kappa-class-ramified-general} we have
      \begin{equation*}
        \card{S_d[2]} =
        \begin{cases}
          2^{\eta}   & \text{if }d=0,     \\
          2^{\eta+1} & \text{if }d\geq 1,
        \end{cases}
        \qquad\text{and}\qquad
        \card{\kappa_d[2]} =
        \begin{cases}
          2^0 & \text{if }d=0,     \\
          2^1 & \text{if }d\geq 1.
        \end{cases}
      \end{equation*}
      Again we obtain \cref{ses_card} for all $d\geq 0$.
    \item Here \cref{two-torsion} and \ref{kappa-class-ramified-special} give
      \begin{equation*}
        \card{S_d[2]} =
        \begin{cases}
          2^{\eta}   & \text{if }d=0,1,   \\
          2^{\eta+1} & \text{if }d\geq 2,
        \end{cases}
        \qquad\text{and}\qquad
        \card{\kappa_d[2]} =
        \begin{cases}
          2^0 & \text{if }d=0,     \\
          2^1 & \text{if }d=1,     \\
          2^2 & \text{if }d\geq 2.
        \end{cases}
      \end{equation*}
      We note that \cref{ses_card} holds only for $d=0$.
      \qedhere
  \end{enumerate}
\end{proof}

We obtain the following as an immediate consequence of \cref{ses-card-split,ses-card-holds}:
\begin{prop}\label{ses-splitting}
  Let $\ell=2$, let $\oh_0$ be an order of odd conductor $c_0$ in an imaginary quadratic field $K$, and let $d\geq 0$.
  \begin{enumerate}[label=(\alph*), ref=Proposition~\theprop(\alph*)]
    \item\label{ses-splitting-split-inert} Suppose $2$ is split or inert in $\oh_0$.
      If ($d\leq 3$) or ($d>3$ and $\Cl(\oh_0)$ has no elements of order $4$), then the sequence \eqref{ses_two} is split.
    \item\label{ses-splitting-ramified} Suppose $2$ is ramified in $\oh_0$ and $D_K\equiv 8\mod{16}$.
      If ($d\leq 1$) or ($d>1$ and $\Cl(\oh_0)$ has no elements of order $4$), then the sequence \eqref{ses_two} is split.
  \end{enumerate}
\end{prop}

\begin{rem}
  The conditions in \cref{ses-splitting} are necessary, as we can see from the following example:
  let $K=\QQ(\sqrt{-39})$ and take $\oh_0=\oh_K$.
  Then $\Cl(\oh_0)\cong\ZZ/4\ZZ$.
  Take $d=4$, so that $\kappa_4\cong\ZZ/4\ZZ\times\ZZ/2\ZZ$.
  One can check (for instance, using Magma \cite{magma}) that $\Cl(\oh_4)\cong \ZZ/8\ZZ\times\ZZ/2\ZZ\times\ZZ/2\ZZ$, so certainly the sequence does not split in this case.

\end{rem}

\subsection{The class number of $\QQ(\sqrt{-\ell})$ is not divisible by $\ell$}

Fix a prime $\ell$.

The aim of this section is to prove the statement in the title; this will be used in \cref{thm:converse_R1} to prove the non-existence of $k$-explosive primes for $(R_1,\ell,d)$ if $k$ is highly divisible by $\ell$.

From Ramar\'e's explicit bounds on values of Dirichlet $L$-functions in~\cite{ramare}, we can deduce the following explicit upper bounds on the class number:
\begin{lem}\label{class-num-upper-bound}
  Let $\oh_K$ be the ring of integers in an imaginary quadratic field $K$.
  Let $h$ be the class number of $\oh_K$ and $D_K$ its discriminant.
  Suppose that $D_K < -4$.
  Define
  \begin{equation*}
    \mathcal{C} =
    \begin{cases}
      2 & \text{if }D_K \equiv 1 \pmod{8}  \\
      4 & \text{if }D_K \equiv 0 \pmod{4}  \\
      6 & \text{if }D_K \equiv 5 \pmod{8}.
    \end{cases}
  \end{equation*}
  Then
  \begin{equation*}
    h \leq \frac{1}{\mathcal{C}\pi} \,\sqrt{|D_K|}\, \big(\log |D_K| + 4.2\big).
  \end{equation*}
\end{lem}
\begin{proof}
  Because of the assumption $D_K < -4$, the group of units $\oh_K^\times$ has cardinality $2$.
  Therefore Dirichlet's class number formula (see \cite[Equation (15) in Chapter 6]{davenport-mnt-2000}) gives
  \begin{equation*}
    h = \frac{\sqrt{|D_K|}}{\pi}\,L\left(1, \chi\right),
  \end{equation*}
  where $\chi=\big(\frac{D_K}{\cdot}\big)$ is the quadratic character associated with $K$.
  Note that, since $D_K<0$, $\chi$ is an odd character, that is $\chi(-1)=-1$.
  Therefore by \cite[Corollary 2]{ramare},
  \begin{equation*}
    L\left(1,\chi\right) \leq \frac{1}{4\left\lvert 1-\frac{\chi(2)}{2} \right\rvert}\, \big(\log|D_K| + 5-2\log(3/2)\big).
  \end{equation*}
  We conclude by noting that $5-2\log(3/2)<4.2$ and
  \begin{equation*}
    4\left\lvert 1 - \frac{\chi(2)}{2} \right\rvert =
    \begin{cases}
      2 & \text{if }D_K \equiv 1 \pmod{8}  \\
      4 & \text{if }D_K \equiv 0 \pmod{4}  \\
      6 & \text{if }D_K \equiv 5 \pmod{8}.
    \end{cases}
  \end{equation*}
  and
  \begin{equation*}
    5-2\log(3/2)<4.2.\qedhere
  \end{equation*}
\end{proof}
\begin{rem}
  The cases left out by \cref{class-num-upper-bound} are $K=\QQ(\sqrt{-1})$ and $K=\QQ(\sqrt{-3})$, both with class number $h=1$.
\end{rem}

\begin{cor}\label{ell-not-divide-class-num}
  If $\ell$ is prime, then it does not divide the class number $h$ of the ring of integers of $\QQ\left(\sqrt{-\ell}\right)$.
\end{cor}
\begin{proof}
  For $\ell = 2,3$, we know that $h = 1$; for $\ell=5$, we have $h=2$.

  For $\ell \geq 7$, we have $|D_K|\leq 4\ell$ so by \cref{class-num-upper-bound}
  \begin{equation*}
    h \leq \frac{1}{2\pi}\,\sqrt{4\ell}\,\big(\log(4\ell)+4.2\big)
    \leq \frac{1}{\pi}\,\sqrt{\ell}\,\big(\log(\ell)+5.6\big).
  \end{equation*}
  Letting
  \begin{equation*}
    g(x) = x-\frac{1}{\pi}\,\sqrt{x}\,\big(\log(x)+5.6\big),
  \end{equation*}
  we have
  \begin{equation*}
    g'(x) = \frac{\pi\sqrt{x}-2}{2x} > 0\quad\text{for }x>\frac{2}{\pi},
  \end{equation*}
  and $g(7)\approx 0.766>0$, so $g(x)>0$ for all $x\geq 7$, hence $h < \ell$ for $\ell\geq 7$.
\end{proof}

\section{A characterisation of solvability}\label{sect:solvability}

We first assemble some useful results about elliptic curves with complex multiplication.

Let $K$ be an imaginary quadratic field with maximal order $\oh_K$.

\begin{lem}
  Let $\oh$ be an order in $K$ and $p$ a prime number that is split in $K$ and coprime to the conductor of $\oh$.
  Let $H$ be the ring class field of $\oh$ and let $\qq$ be a prime ideal of $\oh_H$ that lies above $p$.
  Let $\pp = \qq \cap \oh_K$.
  Let $E$ be an elliptic curve over $\fpbar$ with complex multiplication by $\oh$.
  Then:
  \begin{enumerate}[label=(\alph*), ref=Lemma~\thelem(\alph*)]
    \item The $j$-invariant of $E$ lies in $\FF_{p^f}$, where $f$ is the residue degree of $\qq$.
    \item The residue degree $f$ of $\qq$ equals the order of the class $[\pp\cap\oh]$ in $\Cl(\oh)$, which equals the order of the Frobenius element $\sigma_\pp$ in $\Gal(H/K)$.
    \item\label{j_inv_residue_deg-divisibility_criterion} For any $k\geq 1$, we have $j(E)\in \FF_{p^k}$ if and only if the order of $\sigma_\pp$ in $\Gal(H/K)$ divides $k$.
  \end{enumerate}
\end{lem}
\begin{proof}
  \
  \begin{enumerate}
    \item
      Write $\oh=\ZZ+c\oh_K=\ZZ+c\omega_K\ZZ$, with $p\nmid c$.
      Given an elliptic curve $E$ over $\fpbar$ as in the statement, let $\alpha\in\End(E)\cong\oh$ correspond to $c\omega_K\in\oh$.
      By Deuring's Lifting Theorem (see \cite[Theorem 14 in Chapter 13]{lang}), there exists an elliptic curve $\Etilde$ defined over a number field, an endomorphism $\alphatilde\in\End(\Etilde)$, and a prime ideal $\qq$ above $p$, such that the reduction of $(\Etilde,\alphatilde)$ modulo $\qq$ is isomorphic to $(E,\alpha)$.
      By \cite[Theorem 12 in Chapter 13]{lang} we have
      \begin{equation*}
        \End(\Etilde)\cong\End(E)\cong\oh.
      \end{equation*}

      Therefore the $j$-invariant $j(\Etilde)\in\oh_H$, by \cite[Theorems 5.2.3 and 6.1.2]{schertz}, so we conclude that
      \begin{equation*}
        j(E)\in \oh_H/\qq = \FF_\qq \cong \FF_{p^f}.
      \end{equation*}
    \item
      Consider the decomposition subgroup
      \begin{equation*}
        D_\qq = \{\sigma \in \Gal(H/K)\st \sigma(\qq)=\qq\}.
      \end{equation*}
      Let $\FF_\qq=\oh_H/\qq$ and $\FF_\pp=\oh_K/\pp$.
      Since $\pp$ is unramified, there is a group isomorphism (see, for instance, \cite[Proposition I.9.4, I.9.5]{neukirch})
      \begin{equation*}
        D_\qq \xrightarrow{\,\sim\,} \Gal(\FF_\qq/\FF_\pp)\cong \ZZ/f\ZZ,
      \end{equation*}
      where $f$ is the residue degree of $\qq$.

      Consider the Frobenius element $\sigma_\pp$ of $\Gal(H/K)$ at $\qq$ (as $H$ is an abelian extension, this depends only on the underlying prime $\pp$).
      Under the above isomorphism, it is sent to a generator of $\ZZ/f\ZZ$, so it must have order $f$.
      On the other hand, under the Artin isomorphism $\Gal(H/K)\cong\Cl(\oh)$, the Frobenius element corresponds to the class $[\pp\cap\oh]$ in $\Cl(\oh)$.
      Thus the order of $[\pp\cap\oh]$ is also $f$.
    \item
      By part (a), $j(E)\in\FF_{p^f}$, so $j(E)\in\FF_{p^k}$ if and only if $f\mid k$.
      The rest follows from (b).\qedhere
  \end{enumerate}
\end{proof}

Now, we can give a condition for a volcano to exist.

Given an order $\oh_0$ of conductor $c$ in the imaginary quadratic field $K$, a prime $\ell\ndivides c$, and an integer $d\geq 0$, we use the following notation:
$\oh_d$ denotes the order of conductor $c \ell^d$,
$H_d$ denotes the ring class field of $\oh_d$,
$G_d$ denotes the Galois group $\Gal(H_d / K)$ which is isomorphic to the ideal class group $\Cl(\oh_d)$,
and $\pi_d$ denotes the surjection $G_{d+1} \to G_d$.

\begin{thm}\label{thm:existence}
  Suppose the discriminant of $K$ is $<-4$.
  Let $V_0$ be the crater type corresponding to $\oh_0$ and $\ell$ as in \cref{craters_orders}.
  Fix $d\geq 0$ and consider the volcano $V=(V_0,\ell,d)$.
  Fix $k \geq 1$.
  Define the sets
  \begin{equation*}
    X = \big\{ \sigma\in G_{d+1} \st \ord(\sigma) \ndivides k \text{ and } \ord(\pi_d(\sigma)) \divides k \big\}
  \end{equation*}
  and
  \begin{equation*}
    \mathcal{P} = \big\{p\text{ prime}\st \text{$p\neq\ell$, $p$ split in $\oh_K$, and $V$ appears in $\gpk$ with crater order $\oh_0$}\big\}.
  \end{equation*}
  \begin{enumerate}
    \item If $\card{X}>0$ then $\mathcal{P}$ has positive density in the set of rational primes.
    \item If $X=\emptyset$ then $\mathcal{P}=\emptyset$.
  \end{enumerate}
\end{thm}
\begin{proof}
  Suppose $p\neq \ell$ is a prime that splits in $\oh_K$: $p \oh_K = \pp \overline{\pp}$, $\pp \neq \overline{\pp}$.
  The existence of the volcano $V$ over $\fpk$ with crater order $\oh_0$ is equivalent to having an elliptic curve over $\fpk$ with CM by $\oh_d$ and \emph{no} elliptic curves over $\fpk$ with CM by $\oh_{d+1}$.
  By \ref{j_inv_residue_deg-divisibility_criterion}, this is equivalent to $\sigma_\pp\in G_d$ having order dividing $k$ and $\sigma'_{\pp}\in G_{d+1}$ having order \emph{not} dividing $k$.

  By definition, the natural surjection $\pi_d\st G_{d+1} \to G_d$ maps $\sigma_\pp'$ to $\sigma_\pp$.
  Therefore the existence of the volcano $V$ over $\FF_{p^k}$ is equivalent to $\sigma_\pp'\in X\subseteq G_{d+1}$, where $X$ is the set defined in the statement.

  The set $X$ is trivially stable under conjugation, since $G_{d+1}$ is an abelian group.
  The Chebotarev density theorem tells us that the set
  \begin{equation*}
    S = \big\{\pp\text{ prime of }K\st \pp \text{ is unramified in }H_{d+1}\text{ and }\sigma_\pp'\in X\big\}
  \end{equation*}
  has density
  \begin{equation}\label{eq:galsub}
    \frac{\card{X}}{\card{G_{d+1}}} = \frac{\card{X}}{h(\oh_{d+1})}.
  \end{equation}

  Therefore for a prime $p$ to exist (and hence infinitely many such primes), we must have $\card{X} \neq 0$.
\end{proof}

\begin{rem}
  We could have done this all with $H_{d+1}$ an extension of $\QQ$ instead of $K$ --- that way the set of primes would simply be the ones in $\ZZ^+$.
  We would have then ended up with an extra factor of $\frac{1}{2}$ in \cref{eq:galsub}, since that is the density of split primes.
  But it will be easier to work with $\Gal(H_{d+1}/K)$ as it is isomorphic to $\Cl(\oh_{d+1})$.
\end{rem}

Recall the concept of compatible order from \cref{compatible_order}.

\begin{cor}
  Let $V=(V_0,\ell,d)$ be volcano and let $k\geq 1$.
  \begin{enumerate}[label=(\alph*), ref=Corollary~\thecor(\alph*)]
    \item\label{cor:existence} Suppose there exists an order $\oh_0$ that is compatible with $(V_0,\ell)$ and such that the set $X$ defined in \cref{thm:existence} is non-empty.
      Then $V$ has a positive-density set of \explosive\ primes.

    \item\label{cor:non-existence} Suppose that for all orders $\oh_0$ that are compatible with $(V_0,\ell)$, the set $X$ defined in \cref{thm:existence} is empty.
      Then $V$ has no \explosive\ primes.
  \end{enumerate}
\end{cor}

\begin{lem}
  When $k = 1$, $\card{X} = 0$ if and only if $\ell = 2$ splits in $\oh_K$ and $d = 0$.
\end{lem}
\begin{proof}
  In the case of $k = 1$, $\card{X} = \card{\ker\pi_d} - 1 = \card{\Gal(H_{d+1}/H_d)} - 1$.
  By applying the class number formula it is easy to see that $\card{\Gal(H_{d+1}/H_d)}=\ell - (D_K/\ell)$ when $d = 0$ and $\ell$ otherwise.
\end{proof}
\begin{rem}
  This result, combined with \cref{thm:existence}, recovers \cite[Proposition 4.3]{BCP}.
\end{rem}

In general, the set $X$ in \cref{thm:existence} is tedious to calculate explicitly.
The following criterion, which we state for arbitrary abelian groups (written additively), can assist in deciding whether $X$ is empty or not.

\begin{lem}\label{lem:xequiv}
  Let $k\geq 1$ and let $\pi\st G\to H$ be a homomorphism of finite abelian groups.
  Consider the set
  \begin{equation*}
    X = \{g\in G \st \ord(g) \ndivides k\text{ and }\ord(\pi(g)) \divides k\}.
  \end{equation*}
  \begin{enumerate}[label=(\alph*), ref=Lemma~\thelem(\alph*)]
    \item $X\neq\emptyset$ if and only if $\ker\pi \cap k G\neq \{0\}$.
    \item\label{lem:xequiv-sylow-hated-by-max} Suppose $\ell$ is a prime such that $\ker\pi$ is an $\ell$-group.
      Let $G'$, resp.\ $H'$ be the Sylow $\ell$-subgroup of $G$, resp.\ $H$.
      Let $\pi'\st G'\to H'$ be the restriction of $\pi$ to $G'$, and consider the set
      \begin{equation*}
        X' = \{s\in G'\st \ord(s) \ndivides k\text{ and }\ord(\pi'(s)) \divides k\}.
      \end{equation*}
      Then
      \begin{equation*}
        X\neq\emptyset\quad\Leftrightarrow\quad X'\neq\emptyset\quad\Leftrightarrow\quad \ker\pi'\cap \ell^r G'\neq \{0\},
      \end{equation*}
      where $r=\nu_\ell(k)$.
    \item\label{lem:xequiv-enjoyed-by-max} Let $f\st G\to H$ be a surjective homomorphism of finite abelian $\ell$-groups such that $\ker f$ is cyclic of order $\ell$.
      Then $X\neq\emptyset$ if and only if $\card{G/\ell^r G}=\card{H/\ell^r H}$.
  \end{enumerate}
\end{lem}
\begin{proof}
  \
  \begin{enumerate}
    \item
      For $g\in G$, $\ord(g) \ndivides k$ simply means $k g \neq 0$.
      And $\ord(\pi(g)) \divides k$ means $\pi(k g) = 0$, that is $k g\in\ker\pi$.
      Combining these gives the set $\left(\ker\pi \cap k G\right) \setminus \{ 0\}$.
    \item
      It is clear that $X'\subseteq X$, so that $X'\neq\emptyset\Rightarrow X\neq\emptyset$.

      Now let $g\in X$, so that $k g\neq 0$ and $\pi(k g)=0$.
      Since $G'$ is the Sylow $\ell$-subgroup of the abelian group $G$, we have $G=G'\times A$, where $A$ has order coprime to $\ell$ and $G'\cap A=\{1\}$.
      Write $g=s+a$ with $s\in G'$ and $a\in A$.
      Then
      \begin{equation*}
        0=\pi(k g)=\pi(k s)+\pi(k a),
      \end{equation*}
      so $\pi(k s)$ is the inverse of $\pi(k a)$.
      But the order of $\pi(k s)$ is a power of $\ell$, while the order of $\pi(k a)$ is coprime to $\ell$, so we must have $\pi(k s)=\pi(k a)=0$.
      Hence $k a\in\ker\pi\subseteq G'$ as $\ker\pi$ is an $\ell$-group.
      Since $G'\cap A=\{0\}$ we conclude that $k a=0$, so $k g=k s$.
      We have $s\in G'$, $k s\neq 0$ and $k \pi(s)=0$, therefore $s\in X'$.

      Applying part (a) to $\pi'\st G'\to H'$ and $X'$, we have that $X'\neq\emptyset \Leftrightarrow \ker\pi'\cap k G'\neq \{0\}$.

      But $G'$ is a finite abelian group of $\ell$-power order, so if $\ell\ndivides m$ then the $m$-multiplication map $G'\to m G'$ is an injective group homomorphism, hence $m G=G'$.
      This implies that $k G'=\ell^r G'$.
    \item
      It follows from part (a) and (b) that $X\neq\emptyset\Leftrightarrow\ker f\cap \ell^r G\neq\{0\}$.
      Since $\ker f\cong\ZZ/\ell\ZZ$, we have that $\ker f\cap \ell^r G\neq\{0\}\Leftrightarrow\ker f\subseteq \ell^r G$.

      Now consider the surjective map $\bar{f}\st G/\ell^r G\to H/\ell^r H$ given by $g + \ell^r G\mapsto f(g) + \ell^r H$.
      We can identify the kernel with $\Lambda\coloneq(\ker(f) + \ell^r G)/\ell^r G$.
      If $g + \ell^r G\in\ker\bar{f}$, $g\in G$ then $f(g)\in \ell^r H$, and so $f(g) = \ell^r h$ for some $h\in H$.
      Let $x\in G$ be a pre-image of $h$ so that $f(g) = \ell^r f(x) = f(\ell^r x)$.
      Thus $f(g - \ell^r x) = 0\in\ker f$ and $g + \ell^r G=(g - \ell^r x) + (\ell^r x) + \ell^r G\in\Lambda$.
      Conversely, let $k + \ell^r g + \ell^r G\in\Lambda$ where $k\in\ker f$, $g\in G$.
      Then $k + \ell^r g + \ell^r G=k + \ell^r G\in\ker\bar{f}$ as $\bar{f}(k + \ell^r G) = f(k) + \ell^r H = \ell^r H$.

      By the second isomorphism theorem, $\ker\bar{f}=\Lambda\cong\ker f/(\ker f\cap \ell^r G)$.
      On the other hand, applying the first isomorphism theorem gives $(G/\ell^r G)/\ker\bar{f}\cong H/\ell^r H$.
      Taking orders we have that
      \begin{equation*}
        \card{G/\ell^r G}\cdot\card{(\ker f\cap \ell^r G)} = \card{\ker f}\cdot\card{H/\ell^r H}.
      \end{equation*}
      The result follows from the equivalence $\ker f\subseteq \ell^r G\Leftrightarrow\card{(\ker f\cap \ell^r G)}= \card{\ker f}$.
      \qedhere
  \end{enumerate}
\end{proof}

We return to the setting of \cref{thm:existence}.
For any $d\geq 0$, we let $S_d$ be the Sylow $\ell$-subgroup of $G_d=\Cl(\oh_d)$.
We let $\pi_d$ be the surjection $S_{d+1}\to S_d$, we let $\kappa_d$ be the kernel of the surjection $S_d\to S_0$, and we let $\pi'_d\st\kappa_{d+1}\to\kappa_d$ be the restriction of $\pi_d$.
By \cref{kerpid}, $\ker\pi'_d$ is an $\ell$-group, of order $\ell$ if $d\geq 1$.

By \ref{lem:xequiv-sylow-hated-by-max} we have
\begin{cor}\label{cor:kappa_and_x}
  Let $d>0$, $k\geq 1$, $r=\nu_\ell(k)$.
  If $\ker\pi'_d\cap \ell^r\kappa_{d+1}\neq\{0\}$, then $X\neq\emptyset$.
\end{cor}

By \ref{lem:xequiv-enjoyed-by-max} we have
\begin{cor}\label{cor:nonempty_x_iff}
  Let $d>0$, $k\geq 1$, $r=\nu_\ell(k)$, and suppose $\card\ker\pi'_d=\ell$.
  Then $X\neq\emptyset$ if and only if $\card{S_{d+1}/\ell^r S_{d+1}}=\card{S_d/\ell^rS_d}$.
\end{cor}

\section{Appearance of volcanoes}\label{sect:appearance}

We apply the criteria described in \cref{sect:solvability} to unconditionally prove the appearance or non-appearance of volcanoes of different types over $\fpk$.

\subsection{Craters}\label{subsect:craters}

\begin{lem}\label{depth0x}
  Let $d=0$ and $k\geq 1$.
  In the following cases we have $\card{X}>0$:
  \begin{enumerate}[label=(\alph*), ref=Lemma~\thelem(\alph*)]
    \item\label{depth0x-split} $\ell$ is split in $\oh_K$ and $(\ell - 1) \ndivides k$ (in particular $\ell \neq 2$);
    \item\label{depth0x-inert} $\ell$ is inert in $\oh_K$ and $(\ell + 1) \ndivides k$;
    \item\label{depth0x-ramified} $\ell$ is ramified in $\oh_K$ and $\ell \ndivides k$.
  \end{enumerate}
\end{lem}
\begin{proof}
  Since $d=0$, we are looking at the map $\pi_0\st G_1\to G_0$.
  The kernel $\ker\pi_0$ is denoted $\lambda_1$ in \cref{lambda_1} and described as one of the cyclic groups $\ZZ/(\ell-1)\ZZ$, $\ZZ/(\ell+1)\ZZ$, resp. $\ZZ/\ell\ZZ$, when $\ell$ is split, inert, resp. ramified.
  In each case, letting $g$ be a generator of $\ker\pi_0$, we have that $\ord(g)\ndivides k$ and $\ord(\pi_0(g))=1\divides k$, so $g\in X$.
\end{proof}

\begin{thm}\label{existence_depth_0}
  Fix a prime $\ell$ and an integer $k\geq 1$.
  In each of the following cases, the volcano $(V_0,\ell,0)$ has infinitely many \explosive\ primes:
  \begin{enumerate}[label=(\alph*), ref=Theorem~\thethm(\alph*)]
    \item\label{existence_depth_0-inert} $V_0=I_1$ and $(\ell+1)\ndivides k$;
    \item\label{existence_depth_0-ramified1} $V_0=R_1$ and $\ell\ndivides 3k$;
    \item\label{existence_depth_0-ramified2} $V_0=R_2$ and $\ell\ndivides k$;
    \item\label{existence_depth_0-split} $V_0=S_n$ and $(\ell-1)\ndivides k$.
  \end{enumerate}
\end{thm}

\begin{rem}
  The trick is to find an appropriate order for each case.
  In some sense for $k > 1$ the depth $d = 0$ case is much like the depth $d > 0$ case when $k = 1$.
  The orders we find are identical or quite similar to the ones given in \cite[\S4. Depth $d > 0$]{BCP}.
\end{rem}

\begin{proof}[Proof of \cref{existence_depth_0}]
  For each crater type $V_0$ and prime $\ell$, we find a maximal order $\oh_K$ satisfying the corresponding condition in \cref{depth0x}.
  \begin{enumerate}
    \item (Same as case (1) in \cite[\S4. Depth $d>0$]{BCP}.)\\
      Suppose $V_0 = I_1$.
      By Dirichlet's theorem on arithmetic progressions, there are infinitely many imaginary quadratic fields where $\ell$ is inert.
      Let $\oh_K$ be the ring of integers of any one of them, and apply \ref{depth0x-inert}.
    \item
      Suppose $V_0 = R_1$.
      Take $\oh_K$ to be a corresponding maximal order given by \ref{disc_upper_bound-ramified}.
      Then apply \ref{depth0x-ramified}.
    \item (Same as case (4) in \cite[\S4. Depth $d>0$]{BCP}.)\\
      Suppose $V_0 = R_2$.
      Let $q\neq\ell$ be prime such that $q\geq 5$.
      Let $\oh_K$ be the ring of integers in $K=\QQ(\sqrt{-\ell q})$, then $\ell\oh_K=\LL^2$ is ramified but \ref{disc_upper_bound-ramified} shows that $\LL$ is not principal since the discriminant of $\oh_K$ satisfies $|D|\geq q\ell > 4\ell$.
      Then apply \ref{depth0x-ramified}.
    \item
      Fix $n\geq 1$ and suppose $V_0 = S_n$.
      Since $(\ell-1)\ndivides k$, $\ell$ is odd.
      Let $K_1=\QQ(\sqrt{1-\ell^n})$ and $K_2=\QQ(\sqrt{1-4\ell^n})$.
      By \cite[Theorem 1.4 (3)]{BCP}, either in $\oh_{K_1}$ or in $\oh_{K_2}$, the prime $\ell$ splits into two prime ideals of order $n$ in the class group.
      Let $\oh_K$ be the corresponding ring of integers.
      Then apply \ref{depth0x-split}.
  \end{enumerate}
  In each case $\card{X} \neq 0$, and so the result follows from \cref{thm:existence}.
\end{proof}

\begin{cor}\label{cor:existence_depth_0}
  Fix a crater $V_0$ and $k \geq 1$.
  There exist infinitely many pairs $(p, \ell)$ such that $(V_0, \ell, 0)$ appears in $\gpk$.
\end{cor}

\subsection{Volcanoes of nonzero depth}\label{subsect:non-craters}

\begin{lem}\label{lem:strong_nonemptyx}
  Let $\ell$ be a prime and let $\oh_0$ be an order of conductor coprime to $\ell$.
  Fix $d > 0$ and $k\geq 1$ and let $r = \nu_\ell(k)$.
  Then $\card{X} > 0$ in each of the following cases:

  \begin{center}
    \settabprefix{Lemma~\thelem}
    \setcounter{tabenum}{0}
    \begin{tabular}{llllll}
      \toprule
                                                  & $d$      & $\ell$   & $\ell$ in $\oh_K$ & extra assumptions      & condition on $r$ \\ \midrule
      \tabitem{lem:nonemptyx-split-inert-general} & $\geq 1$ & $\geq 3$ & split or inert    & none                   & $r<d$            \\ \midrule
      \tabitem{lem:nonemptyx-ramified-general}    & $\geq 1$ & $=2$     & ramified          & $D_K\equiv 8\mod{16}$  & $r<d+1$          \\
                                                  &          & $=3$     &                   & $D_K\equiv 3\mod{9}$                      \\
                                                  &          & $\geq 5$ &                   & none                                      \\ \midrule
      \tabitem{lem:nonemptyx-ramified-special}    & $\geq 1$ & $=2$     & ramified          & $D_K\equiv 12\mod{16}$ & $r<d$            \\
                                                  &          & $=3$     &                   & $D_K\equiv 6\mod{9}$                      \\ \midrule
      \tabitem{lem:nonemptyx-split-inert-special} & $\geq 2$ & $=2$     & split or inert    & none                   & $r<d-1$          \\
      \bottomrule
    \end{tabular}
  \end{center}

  The remaining special case (in part (d)) is: for $\ell=2$ split or inert and $d<2$, if $r<1$ then $\card{X}>0$.
\end{lem}
\begin{proof}
  In each case, we extract from the relevant part of \cref{kappa-classification} and \cref{lambda-classification} the information necessary to deduce that $\ker\pi'_d\cap \ell^r \kappa_{d+1}\neq\{0\}$, then use \cref{cor:kappa_and_x}.
  \begin{enumerate}
    \item
      By \ref{kappa-class-split-inert-general}, $\ker\pi'_d$ is generated by $\ell^{d-1}$ as a subgroup of $\kappa_{d+1}\cong \ZZ/\ell^d\ZZ$.
      Since $\ell^r \kappa_{d+1}\cong \ell^r\ZZ/\ell^d\ZZ$, the condition $r<d$ implies that $\ell^{d-1}=\ell^r\ell^{d-1-r}\in\ell^r\kappa_{d+1}$, as required.
    \item
      Same as (a), except that by \ref{kappa-class-ramified-general}, $\ker\pi'_d$ is generated by $\ell^d$ as a subgroup of $\kappa_{d+1}\cong \ZZ/\ell^{d+1}\ZZ$.
      The shift from $\ell^{d-1}$ to $\ell^d$ in comparison to the previous part explains the adjusted condition $r<d+1$.
    \item From \ref{kappa-class-ramified-special}, $\ker\pi'_d$ is generated by $(\ell^{d-1},0)\in \ZZ/\ell^d\ZZ\times \ZZ/\ell\ZZ\cong \kappa_{d+1}$.
      The case $r=0$ is trivial.
      If $r>0$, we have $\ell^r \kappa_{d+1}\cong \ell^r\ZZ/\ell^d\ZZ\times\{0\}$ and we conclude as in (a).
    \item From \ref{kappa-class-split-inert-special}, if $d\geq 2$ then $\ker\pi'_d$ is generated by $(2^{d-2},0)\in\ZZ/2^{d-1}\ZZ\times\ZZ/2\ZZ\cong\kappa_{d+1}$.
      The case $r=0$ is trivial.
      If $r>0$, we have $2^r \kappa_{d+1}\cong 2^r\ZZ/2^{d-1}\ZZ\times\{0\}$, so the condition $r<d-1$ allows us to conclude.

      The special case with $d<2$ uses $\kappa_1\cong 0$ and $\kappa_2\cong\ZZ/2\ZZ$.
      \qedhere
  \end{enumerate}
\end{proof}

\begin{thm}\label{thm:exist_primes}
  Fix a prime $\ell$, $d>0$, and $k\geq 1$.
  Let $r=\nu_\ell(k)$.
  In each of the following cases, the given volcano $(V_0,\ell,d)$ has infinitely many \explosive\ primes.

  \begin{center}
    \settabprefix{Theorem~\thethm}
    \setcounter{tabenum}{0}
    \begin{tabular}{lllll}
      \toprule
                                                        & $V_0$        & $d$      & $\ell$   & condition on $r$ \\ \midrule
      \tabitem{thm:exist_primes-split-inert-general}    & $S_n$, $I_1$ & $\geq 1$ & $\geq 3$ & $r<d$            \\ \midrule
      \tabitem{thm:exist_primes-ramified-non-principal} & $R_2$        & $\geq 1$ & $\geq 2$ & $r<d+1$          \\ \midrule
      \tabitem{thm:exist_primes-ramified-principal}     & $R_1$        & $\geq 1$ & $\neq 3$ & $r<d+1$          \\
                                                        &              &          & $=3$     & $r<d$            \\ \midrule
      \tabitem{thm:exist_primes-split-inert-special}    & $S_n$, $I_1$ & $\geq 2$ & $=2$     & $r<d-1$.         \\
      \bottomrule
    \end{tabular}
  \end{center}

  The remaining special case (in part (d)) is: if $V_0=S_n$ or $I_1$, $\ell=2$, $d=1$, and $r<1$ (that is, $k$ is odd), then the volcano $(V_0,\ell,d)$ has infinitely many \explosive\ primes.
\end{thm}
\begin{proof}
  The strategy is to identify some appropriate order $\oh_0$ and use the relevant part of \cref{lem:strong_nonemptyx} to show that $\card{X}>0$, then conclude by \ref{cor:existence}.

  \begin{enumerate}
    \item Here $\ell\geq 3$.
      For $V_0=I_1$, let $\oh_0=\oh_K$ where $K$ is such that $\ell$ is inert; in other words, the discriminant $D_K$ satisfies $\left(\frac{D_K}{\ell}\right)=-1$.
      For $V_0=S_n$, let $\oh_0=\oh_K$ as in the proof of \ref{existence_depth_0-split}.
      In both cases we are done by \ref{lem:nonemptyx-split-inert-general}.

    \item For $V_0=R_2$ we follow the proof of \ref{existence_depth_0-ramified2}: if $\ell\neq 5$, let $\oh_0=\oh_K$ with $K=\QQ(\sqrt{-5\ell})$.
      For $\ell=5$, take $\oh_0=\oh_K$ with $K=\QQ(\sqrt{-35})$.
      In all cases we are in the situation of \ref{lem:nonemptyx-ramified-general}.

    \item For $V_0=R_1$, take $K=\QQ(\sqrt{-\ell})$ (this is guided by \cref{disc_upper_bound}).
      If $\ell\neq3$, let $\oh_0=\oh_K$ and we are done by \ref{lem:nonemptyx-ramified-general} (when $\ell=2$, we have $D_K=-8\equiv 8\mod{16}$).
      For $\ell=3$, let $\oh_0$ be the order of conductor $2$ in $K$.
      Then $D_K=-3\equiv 6\mod{9}$, so we apply \ref{lem:nonemptyx-ramified-special} instead.

    \item Here $\ell=2$.
      For $V_0=I_1$, let $\oh_0=\oh_K$ where $K$ is such that $2$ is inert; in other words $D_K\equiv 5\mod{8}$.

      For $V_0=S_n$, we follow \cite[Theorem 1.4]{BCP} and let $K=\QQ(\sqrt{-39})$ if $n=4$ and $K=\QQ(\sqrt{1-2^{n+2}})$ if $n\neq 4$, then take $\oh_0=\oh_K$.

      If $d\geq 2$ and $r<d-1$, then \ref{lem:nonemptyx-split-inert-special} gives us $\card{X}>0$ and we are done.

      If $d<2$ and $r<1$, we conclude by using the special case in \ref{lem:nonemptyx-split-inert-special}.
      \qedhere
  \end{enumerate}
\end{proof}

\begin{rem}\label{rem:maximal_orders_for_existence}
  In almost all cases, the order $\oh_0$ described above is a maximal order.
  The only exception is the case $(R_1, 3, d)$, where we need to take $\oh_0$ of conductor $2$.
\end{rem}

In the principally ramified case $R_1$ we can prove a converse to \cref{thm:exist_primes}:
\begin{thm}\label{thm:converse_R1}
  Fix a prime $\ell$, $d>0$, and $k\geq 1$.
  Let $r=\nu_\ell(k)$.

  If $\ell\neq 3$ and $r\geq d+1$, the volcano $(R_1,\ell,d)$ has no \explosive\ primes.

  If $r\geq d$, the volcano $(R_1,3,d)$ has no \explosive\ primes.
\end{thm}
\begin{proof}
  Let $\oh_0$ be any order satisfying \ref{disc_upper_bound-ramified} with discriminant $<-4$, which exhausts all possibilities for $V_0=R_1$.
  We argue that $S_0$, the Sylow $\ell$-subgroup of $\Cl(\oh_0)$, is trivial, so that $\kappa_{d+1}=S_{d+1}$ and we can use \ref{lem:xequiv-sylow-hated-by-max} in the guise: $\ker\pi'_d\cap\ell^r\kappa_{d+1}\neq\{0\}$ if and only if $X\neq\emptyset$.

  In all such cases appearing in \ref{disc_upper_bound-ramified}, $\ell$ does not divide the class number of $\oh_K$: $K=\QQ(\sqrt{-\ell})$, where this claim follows from \cref{ell-not-divide-class-num}.
  So in the cases where $\oh_0=\oh_K$, $\ell$ does not divide the order of $\Cl(\oh_0)$, thus $S_0$ is trivial.

  This leaves the case $\ell\equiv 3\mod{4}$, $K=\QQ(\sqrt{-\ell})$, $\oh_0$ of conductor $2$ in $\oh_K$.
  By the class number formula appearing in \cite[Theorem 7.24]{cox} we have that the class number of $\oh_0$ divides $2$ times the class number of $\oh_K$, so again $\ell$ does not divide the order of $\Cl(\oh_0)$.

  We can now proceed as in the proof of \ref{lem:nonemptyx-ramified-general} and \ref{lem:nonemptyx-ramified-special}.
  \begin{itemize}
    \item If $\ell\geq 5$ then by \ref{kappa-class-ramified-general} we have $\kappa_{d+1}\cong\ZZ/\ell^{d+1}\ZZ$, so $\ell^r\kappa_{d+1}\cong\ell^r\ZZ/\ell^{d+1}\ZZ$; hence if $r\geq d+1$, $\ell^r S_{d+1}=\ell^r \kappa_{d+1}=\{0\}$ so $X=\emptyset$.
    \item If $\ell=3$ then $D_K=-3\equiv 6\mod{9}$, so by \ref{kappa-class-ramified-special} we have $\kappa_{d+1}\cong\ZZ/\ell^d\ZZ\times\ZZ/\ell\ZZ$.
      Hence if $r\geq d$ then $X=\emptyset$.
    \item Finally, if $\ell=2$ then we have $D_K=-8\equiv 8 \mod{16}$ and \ref{kappa-class-ramified-general} gives $\kappa_{d+1}\cong\ZZ/\ell^{d+1}\ZZ$.
      So if $r\geq d+1$ then $X=\emptyset$.
      \qedhere
  \end{itemize}
\end{proof}

In the split case $S_n$ we can prove a weak partial converse to \cref{thm:exist_primes}:
\begin{prop}
  Fix a prime $\ell$, $d>0$, and $k\geq 1$.
  Let $r=\nu_\ell(k)$.
  There exists $N_d \in \NN$ such that if $r \geq N_d$, the volcano $(S_n,\ell,d)$ has no \explosive\ primes.
\end{prop}
\begin{proof}
  By \ref{disc_upper_bound-split}, there are finitely many orders $\oh_0$ compatible with $(S_n, \ell)$, namely those whose discriminant is bounded by $4\ell^n-1$.
  Let $N_d$ be the maximum of the $\ell$-adic valuations of the exponents of the $\ell$-groups $S_{d+1}$ corresponding to $\oh_{d+1}$ for each of these orders; then the claim follows from \cref{lem:xequiv} and \ref{cor:non-existence}.
\end{proof}

We can say a bit more if we restrict to $\ell=2$:
\begin{thm}
  Let $\ell=2$.
  \begin{enumerate}[label=(\alph*), ref=Theorem~\thethm(\alph*)]
    \item\label{thm:split_two_converse_low_depth}
      If ($d=1$ and $r\geq 1$) or ($d=2,3$ and $r \geq d-1$), then the volcanoes $(S_n,2,d)$ and $(I_1,2,d)$ have no \explosive\ primes.
    \item\label{thm:split_two_converse_high_depth}
      Suppose that $n\geq 1$ is such that none of the class groups of the orders $\oh_0$ that are compatible with $(S_n,2)$ contain any elements of order $4$.
      If $d>3$ and $r \geq d-1$, then the volcano $(S_n,2,d)$ has no \explosive\ primes.
  \end{enumerate}
\end{thm}
\begin{proof}
  We start with an observation that applies to both parts (a) and (b).

  Let $\oh_0$ be any order that is compatible with $(V_0,2)$.
  The aim is to show that under the given conditions we have $X=\emptyset$, where $X$ is the set from \ref{cor:non-existence} and \cref{thm:existence}.
  By \cref{lem:xequiv}, this is equivalent to showing that
  \begin{equation*}
    \ker\pi_d\cap 2^r S_{d+1}=\{0\}.
  \end{equation*}
  We now make a crucial assumption: suppose that the short exact sequence
  \begin{equation*}
    0 \to \kappa_m \to S_m \to S_0 \to 0
  \end{equation*}
  is split for $m=d+1$.
  Then by \cref{lem:splitting}, so is the sequence for $m=d$ and we can identify $\ker \pi_d$ with $(\ker \pi'_d)\times \{0\}\subseteq\kappa_{d+1}\times S_0\cong S_{d+1}$.
  Therefore
  \begin{equation*}
    \ker \pi_d \cap 2^r S_{d+1}\subseteq \big((\ker\pi'_d)\times\{0\}\big)\cap \big(2^r \kappa_{d+1}\times 2^r S_0\big)
    =\big(\ker\pi'_d\cap 2^r\kappa_{d+1}\big)\times\{0\}.
  \end{equation*}
  The upshot is that, under the crucial splitting assumption made above, it suffices to prove that $\ker\pi'_d\cap2^r\kappa_{d+1}=\{0\}$.

  We now consider the situation in each part separately:
  \begin{enumerate}
    \item Since $d\leq 3$, we know from \ref{ses-splitting-split-inert} that the short exact sequence splits, so this assumption is satisfied.
      We now use the structure of $\kappa_d$ from \ref{kappa-class-split-inert-special}.

      If $d=1$ we have $\kappa_{d+1}=\kappa_2\cong\ZZ/2\ZZ$, so $2^r \kappa_{d+1}\cong 2^r\ZZ/2\ZZ=0$ if $r\geq 1$.

      If $d=2,3$ we have $\kappa_{d+1}\cong\ZZ/2^{d-1}\ZZ\times\ZZ/2\ZZ$, so $2^r\kappa_{d+1}\cong 2^r\ZZ/2^{d-1}\ZZ\times 2^r\ZZ/2\ZZ=0$ if $r\geq d-1$.

      In all cases we conclude that $X=\emptyset$ for all compatible orders $\oh_0$, so by \ref{cor:non-existence} we deduce that there are no \explosive\ primes.
    \item
      As $\Cl(\oh_0)$ is assumed to not have any elements of order $4$, \ref{ses-splitting-ramified} tells us that the short exact sequence splits.

      By \ref{kappa-class-split-inert-special} we have $\kappa_{d+1}\cong\ZZ/2^{d-1}\ZZ\times\ZZ/2\ZZ$, so we conclude as in part (a).
      \qedhere
  \end{enumerate}
\end{proof}

\begin{rem}
  \cref{prop:counterexample} is the case $d=1$, $n=2$, $k=2$ of \ref{thm:split_two_converse_low_depth}.
\end{rem}

\begin{rem}
  The condition on orders of elements in class groups of orders that we impose in \ref{thm:split_two_converse_high_depth} raises questions of existence/distribution of integers $n\geq 1$ with this property.

  Computationally, it appears that there are only six such values of $n$, namely $\{1,2,3,5,6,7\}$.
\end{rem}

\section{Computations and heuristics}\label{sect:computations}

Fix a prime $\ell$, $d>0$, and $k\geq1$.
Let $r = \nu_\ell(k)$.
The appearance of volcanoes $(V_0,\ell,d)$ listed in the following table is not decided by the results we have described thus far:
\begin{center}
  \begin{tabular}{lllll}
    \toprule
    $V_0$ & $\ell$  & $d$     & $n$                     & Condition on $r$ \\
    \midrule
    $I_1$ & $\ge 3$ & $\ge 1$ & --                      & $r \ge d$        \\
    $I_1$ & $2$     & $\ge 4$ & --                      & $r \ge d-1$      \\
    $R_2$ & $\ge 2$ & $\ge 1$ & --                      & $r \ge d+1$      \\
    $S_n$ & $\ge 3$ & $\ge 1$ & $\geq 1$                & $r \ge d$        \\
    $S_n$ & $2$     & $\ge 4$ & $\notin\{1,2,3,5,6,7\}$ & $r \ge d-1$      \\
    \bottomrule
  \end{tabular}
\end{center}

In the inert and ramified cases, we have identified conditions on a hypothetical order $\oh_0$ that would guarantee the existence of infinitely many \explosive\ primes.
These conditions relate to the structure and rank of the Sylow $\ell$-subgroups of the class groups of $\oh_0$ and some deeper orders $\oh_d$.
(A discussion of the behaviour of these ranks as $d$ varies can be found in Appendix \ref{sec:ranks}.)
\begin{prop}\label{heuristics-justification}
  Let $d > 0$, $k\geq1$ and $r = \nu_\ell(k)$.
  \begin{enumerate}[label=(\alph*), ref=Proposition~\theprop(\alph*)]
    \item\label{heuristics-justification-inert}
      Fix $\ell\geq 3$ and $r\geq d$.
      Suppose $\oh_0$ is an order compatible with $V = (I_1, \ell)$ such that
      \begin{equation}\label{I1_cond}
        S_0\cong\ZZ/\ell^N\ZZ\text{ for some $N\geq 0$ and $\rk(S_2) = \rk(S_0)$.}
      \end{equation}
      Then $X\neq\emptyset$ if and only if $N\geq r+1-d$.
    \item\label{heuristics-justification-ramified}
      Fix $\ell\geq 5$ and $r\geq d+1$.
      Suppose $\oh_0$ is an order compatible with $V = (R_2, \ell)$ such that
      \begin{equation}\label{R2_cond}
        S_0\cong\ZZ/\ell^N\ZZ\text{ for some $N\geq 0$, and $\rk(S_1) = \rk(S_0)$.}
      \end{equation}
      Then $X\neq\emptyset$ if and only if $N\geq r-d$.
  \end{enumerate}
\end{prop}
\begin{proof}\
  \begin{enumerate}
    \item
      With our assumptions, it follows from \cref{kappa-classification} and \cref{class-group-ranks} that $S_d\cong\ZZ/\ell^{N+d-1}\ZZ$ for all $d\geq1$.
      Now
      \begin{equation*}
        \card{S_d/\ell^r S_d} = \begin{cases}
          \ell^{N+d} & \text{if } r\geq N+d-1, \\
          \ell^r     & \text{otherwise.}
        \end{cases}
      \end{equation*}
      We use \cref{cor:nonempty_x_iff} to conclude since $\card{S_{d+1}/\ell^r S_{d+1}} =\card{S_d/\ell^r S_d}$ if and only if $N\geq r+1-d$.
    \item The same as part (a), essentially replacing $d-1$ with $d$ throughout.\qedhere
  \end{enumerate}
\end{proof}

Since for $V_0=I_1$ or $R_2$ there are infinitely many orders $\oh_0$ that are compatible with any fixed $(V_0,\ell)$, it seems reasonable to hope that some of these will satisfy the respective conditions \eqref{I1_cond}, \eqref{R2_cond}, and hence there will be infinitely many \explosive\ primes for all $r$.

But actually constructing compatible orders in this generality appears out of reach currently.
In this section, we describe some conjectures on the distribution of class groups with the desired properties and the conditional results that we can obtain from them.
We also discuss computational evidence in favour of these conjectures.

\subsection{\cohlen\ heuristics}\label{subsect:cohen-lenstra}

In \cite{cohen-lenstra}, Cohen and Lenstra introduced a probability distribution for predicting the frequency of properties satisfied by class groups of number fields.
It is based on the heuristic that an abelian group should occur with probability that is inversely proportional to the number of automorphisms of the group.

In our situation, the formulation is as follows.
Let the sample space $\Omega$ be the set of all isomorphism classes of finite abelian $\ell$-groups and endow $\Omega$ with the discrete $\sigma$-algebra.
For $G\in\Omega$, define the (probability) measure
\begin{equation*}
  \mu_\text{CL}(\{G\}) = \frac{\eta_\infty(\ell)}{\card{\Aut G}}, \quad \text{ where } \quad \eta_\infty(\ell)=\prod_{k=1}^\infty \big(1 - \ell^{-k}\big).
\end{equation*}

Let $K$ denote an imaginary quadratic field with discriminant $D_K$ and ring of integers $\oh_K$.
For a ``reasonable'' random variable $h$, the \cohlen\ heuristics predict that the average of $h$ over the Sylow $\ell$-subgroup of the ideal class group $\Cl(\oh_K)$ is the $0$-average of $h$:
\begin{equation*}
  \lim_{X\to\infty} \frac{\displaystyle \sum_{\left|D_K\right|\leq X} h\big(\Cl(\oh_K)[\ell^\infty]\big)}{\displaystyle\sum_{\left|D_K\right|\leq X} 1} = M_0(h),
\end{equation*}
where the sums are over all imaginary quadratic fields of discriminant $\left|D_K\right|\leq X$.

\begin{rem}\label{local-cond}
  Instead of considering all number fields, one can restrict to a subset given by local conditions (for example, imposing that $\ell$ is inert).
  Wood conjectures (with supporting computations) that the distribution of number fields is invariant under such restrictions, see \cite[Conjecture 1.4]{wood2018}.
\end{rem}

Whilst there is strong numerical evidence in support of the heuristics, very few results have been proved.
One instance is the work of Davenport and Heilbronn (in \cite{dh}, predating Cohen and Lenstra): they show that for $h(G)=\card{G[3]}$, the average of $h$ over the class groups of imaginary quadratic fields is $2$.
Cohen and Lenstra then computed $M_0(h) = 2$, confirming the heuristic's prediction in this case.

Before we move on, we want to know the probability with which cyclic groups of exponent at least $e\geq1$ appear.
\begin{lem}\label{lem:prob}
  Fix $e\geq1$.
  The probability that a group is of the form $\ZZ/\ell^N\ZZ$ where $N\geq e$ is
  \begin{equation*}
    \frac{\eta_\infty(\ell)}{\ell-1}\frac{\ell^{2-e}}{\ell-1}.
  \end{equation*}
\end{lem}
\begin{proof}
  From \cite[Theorem 10.6]{ch-calc}, the probability of the group $\ZZ/\ell^N\ZZ$, $N\geq1$, is
  \begin{equation*}
    \frac{\eta_\infty(\ell)}{\ell-1} \ell^{1-N}.
  \end{equation*}
  So our desired probability is
  \begin{equation*}
    \frac{\eta_\infty(\ell)}{\ell-1}\sum_{i=e}^\infty \ell^{1-i} = \frac{\eta_\infty(\ell)}{\ell-1}\frac{\ell^{2-e}}{\ell-1}.\qedhere
  \end{equation*}
\end{proof}

\subsection{Extending the heuristics}

We modify the \cohlen\ heuristics with the aim of investigating the probability of finding a maximal order $\oh_0=\oh_K$ satisfying the conditions required in \cref{heuristics-justification}.
The modifications enable us to include information about the class groups of orders of index $\ell^d$ in $\oh_K$.
To this end, we take $\Omega$ to be the set of all the pairs $(G, \xi)$ (up to isomorphism), where $G$ is a finite abelian $\ell$-group and $\xi\in\Ext_\ZZ^1(G, \ZZ/\ell\ZZ)$.
We define the (probability) measure
\begin{equation*}
  \mu_\text{CL'}(\{(G, \xi)\}) = \frac{\eta_\infty(\ell)}{\card{\Aut G} \cdot \card{\Ext_\ZZ^1(G, \ZZ/\ell\ZZ)}}.
\end{equation*}

Thus, for an imaginary quadratic field $K$ we take $\oh_0=\oh_K$ and
\begin{itemize}
  \item in the case $V_0=I_1$, $\ell$ is inert in $\oh_0$, and we are considering the pair $(S_0,\xi)$ where $\xi\in\Ext_\ZZ^1(S_0,\kappa_2)$ is the class defined by $\pi'_2$;
  \item in the case $V_0=R_2$, $\ell$ is ramified in $\oh_0$, and we are considering the pair $(S_0,\xi)$ where $\xi\in\Ext_\ZZ^1(S_0,\kappa_1)$ is the class defined by $\pi'_1$.
\end{itemize}
Note that we are using \cref{local-cond} here to assume invariance under local conditions.

The conditions in \eqref{I1_cond} and \eqref{R2_cond} about ranks being equal is then precisely the statement that $\xi\neq0$.
We can now give the probability that the conditions in \cref{heuristics-justification} are satisfied:
\begin{lem}\label{lem:prob-ext}
  Fix $e\geq1$.
  The probability that the pairs $(\ZZ/\ell^N\ZZ, \xi)$ occur, where $\xi\neq0$ and $N\geq e$, is
  \begin{equation*}
    \eta_\infty(\ell) \frac{\ell^{1-e}}{\ell-1}.
  \end{equation*}
\end{lem}
\begin{proof}
  Since $\card{\Ext_\ZZ^1(\ZZ/\ell^N\ZZ, \ZZ/\ell\ZZ)} = \ell$, the proportion of $\xi$ that are non-zero is $(\ell - 1)/\ell$.
  Using \cref{lem:prob}, we conclude that the probability is
  \begin{equation*}
    \left(\frac{\eta_\infty(\ell)}{\ell-1}\frac{\ell^{2-e}}{\ell-1}\right) \left(\frac{\ell-1}{\ell}\right) = \eta_\infty(\ell)\frac{\ell^{1-e}}{\ell-1}.\qedhere
  \end{equation*}
\end{proof}

\begin{thm}
  Assume that the modified \cohlen\ heuristics hold.
  Fix $d > 0$, $k\geq1$, and let $r = \nu_\ell(k)$.
  \begin{enumerate}[label=(\alph*), ref=Theorem~\thethm(\alph*)]
    \item\label{heuristics-inert}
      If $\ell\geq3$ and $r\geq d$, the volcano $(I_1, \ell, d)$ has infinitely many \explosive\ primes.
    \item\label{heuristics-ramified}
      If $\ell\geq5$ and $r\geq d+1$, the volcano $(R_2, \ell, d)$ has infinitely many \explosive\ primes.
  \end{enumerate}
\end{thm}
\begin{proof}\
  \begin{enumerate}
    \item
      Using \cref{lem:prob-ext} with $e = r + 1 - d$, the probability that an imaginary quadratic field $K$ where $\ell$ is inert has ideal class group satisfying \eqref{I1_cond} is
      \begin{equation*}
        \eta_\infty(\ell)\frac{\ell^{d-r}}{\ell-1}.
      \end{equation*}
      Therefore, for any $r\geq d$, we can find an order compatible with $(I_1, \ell, d)$ by \cref{heuristics-justification}.
      The result then follows from \ref{cor:existence}.
    \item
      Similar to part (a), again replacing $d-1$ with $d$ throughout.
      Note that we do not distinguish between $R_1$ and $R_2$ when considering our local conditions, but this has a negligible effect on the probabilities as, by \ref{disc_upper_bound-ramified}, there is only one field where $\ell$ ramifies principally.\qedhere
  \end{enumerate}
\end{proof}

\subsection{Computations}

To see if our modified heuristics are ``reasonable'', we establish some computational evidence supporting them.
We build on the computations done in \cite{clgrp} to count the number of imaginary quadratic fields $K$ with discriminant $\abs{D_K}<2^{37}$ such that $\oh_0=\oh_K$ satisfies \eqref{I1_cond}, \eqref{R2_cond} for various $\ell, r$ and $d$.
Our code is available at \cite{code}; we used the University of Melbourne's High Performance Computing system ``Spartan'' \cite{spartan} to perform our computations.

\begin{figure}[h!]
  \centering
  \subfloat[Values of $i_{\ell, e}(x)$.]{\includegraphics[page=1]{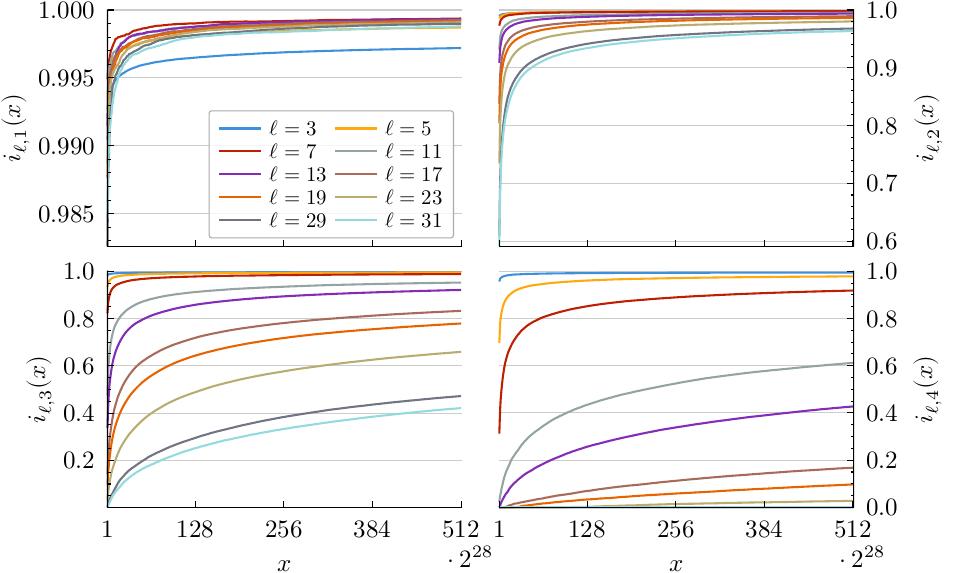}\label{fig:inert}} \\
  \subfloat[Values of $r_{\ell, e}(x)$.]{\includegraphics[page=2]{graphs.pdf}\label{fig:ramified}}
  \caption{Computational evidence for the modified \cohlen\ heuristics.}
\end{figure}

To present the results we introduce two functions.
For $e = r + 1 - d\geq1$ define $i_{\ell, e}(x)$ to be the proportion of imaginary quadratic fields $K$ with discriminant $\abs{D_K}<x$ where $\oh_K$ is compatible with $V = (I_1, \ell)$, such that \eqref{I1_cond} is satisfied, divided by the probability given in \cref{lem:prob-ext}.
We define $r_{\ell, e}(x)$ (for $e = r - d$) analogously for $V = (R_2, \ell)$, instead satisfying \eqref{R2_cond}.

If our heuristics hold, we would expect both of these functions to approach $1$ as $x$ grows.
We observe this behaviour (albeit with slow convergence for larger primes $\ell$ when $e = 3,4$) in \cref{fig:inert,fig:ramified}.

\appendix

\section{Relating the class groups of nested orders}
Let $K$ be an imaginary quadratic field.
Let $\oh_m$ be the order of $K$ of conductor $m\in\ZZ_{\geq 1}$.

The ring map $\ZZ\to\oh_K\to\oh_K/m\oh_K$ gives rise to an injective group homomorphism
\begin{equation*}
  (\ZZ/m\ZZ)^\times \hookrightarrow (\oh_K/m\oh_K)^\times.
\end{equation*}
Let $\kappa_m$ denote its cokernel.
By \cite[Equations (7.25) and (7.27)]{cox}, $\kappa_m$ is the kernel of the canonical surjective homomorphism $\Cl(\oh_m)\to \Cl(\oh_K)$; in other words we have two short exact sequences
\begin{equation}\label{eq:ses1}
  1\to (\ZZ/m\ZZ)^\times \to (\oh_K/m\oh_K)^\times \to \kappa_m\to 1
\end{equation}
and
\begin{equation}\label{eq:ses2}
  1\to \kappa_m\to \Cl(\oh_m)\to \Cl(\oh_K)\to 1.
\end{equation}

Now let $n\in\ZZ_{\geq 1}$ be coprime to $m$ and consider the orders $\oh_n$ and $\oh_{mn}$.
The following diagram allows us to identify the kernel of the surjective map $\kappa_{mn}\to \kappa_m$ with $\kappa_n$:
\begin{equation*}
  \begin{tikzcd}[column sep=2em,row sep=1.5em,trim left=-8cm]
    & 1 \arrow[d] & 1 \arrow[d] & 1 \arrow[d] & \\
    1 \arrow[r]
    & (\ZZ/n\ZZ)^\times \arrow[r] \arrow[d]
    & (\oh_K/n\oh_K)^\times \arrow[r] \arrow[d]
    & \kappa_n \arrow[d] \arrow[dddll, fashionfuchsia, dotted, out=0, in=180, looseness=2.2, "" description]
    & \\
    1 \arrow[r]
    & (\ZZ/mn\ZZ)^\times \arrow[r] \arrow[d]
    & (\oh_K/mn\oh_K)^\times  \arrow[r] \arrow[d]
    & \kappa_{mn} \arrow[r] \arrow[d]
    & 1 \\
    1 \arrow[r]
    & (\ZZ/m\ZZ)^\times \arrow[r] \arrow[d]
    & (\oh_K/m\oh_K)^\times  \arrow[r] \arrow[d]
    & \kappa_m \arrow[r] \arrow[d]
    & 1 \\
    & 1
    & 1
    & 1
  \end{tikzcd}
\end{equation*}
One should imagine building this diagram by starting with the short exact sequence \eqref{eq:ses1} for $mn$, then placing the one for $m$ underneath and taking the two leftmost vertical maps to be the canonical surjections.
This induces the third vertical map $\kappa_{mn}\to\kappa_m$.
Then it is a matter of determining the kernels of the vertical maps, which are clear for the two left maps, and follows from the Snake Lemma for the third.

We now put together the short exact sequences \eqref{eq:ses2} for $mn$ and for $m$ into another diagram:
\begin{equation*}
  \begin{tikzcd}[column sep=2em,row sep=1.5em,trim left=-5cm]
    & 1 \arrow[d] & 1 \arrow[d] &  & \\
    1 \arrow[r]
    & \kappa_n \arrow[r] \arrow[d]
    & \phantom{\kappa_n} \arrow[r] \arrow[d]
    & 1 \arrow[d]
    & \\
    1 \arrow[r]
    & \kappa_{mn} \arrow[r] \arrow[d]
    & \Cl(\oh_{mn})  \arrow[r] \arrow[d]
    & \Cl(\oh_K) \arrow[r] \arrow[d, "="]
    & 1 \\
    1 \arrow[r]
    & \kappa_m \arrow[r] \arrow[d]
    & \Cl(\oh_m)  \arrow[r] \arrow[d]
    & \Cl(\oh_K) \arrow[r] \arrow[d]
    & 1 \\
    & 1
    & 1
    & 1
  \end{tikzcd}
\end{equation*}
We know from above that the kernel in the first column is $\kappa_n$, and the fact that the rightmost vertical map is the identity on $\Cl(\oh_K)$ allows us to identify the missing kernel in the second column with $\kappa_n$.

We have therefore proved:
\begin{prop}\label{kappa_n}
  The kernel of the canonical surjective map $\Cl(\oh_{mn})\to\Cl(\oh_m)$ can be identified with $\kappa_n$, in other words that there is a short exact sequence
  \begin{equation*}
    1 \to \kappa_{n} \to \Cl(\oh_{mn}) \to \Cl(\oh_m) \to 1,
  \end{equation*}
  where $\kappa_n$ is defined by
  \begin{equation*}
    1\to (\ZZ/n\ZZ)^\times \to \big(\oh_K/n\oh_K\big)^\times \to \kappa_n\to 1.
  \end{equation*}
\end{prop}

Note that we may replace the middle group in this last sequence by $\big(\oh_m/n\oh_m\big)^\times$:
\begin{lem}\label{order_mod_integer}
  Given that $\oh_m$ has conductor $m$ coprime to $n$, the inclusion $\oh_m\hookrightarrow\oh_K$ induces a ring isomorphism
  \begin{equation*}
    \oh_m/n\oh_m \cong \oh_K/n\oh_K.
  \end{equation*}
  In particular, we have a group isomorphism $(\oh_m/n\oh_m)^\times\cong (\oh_K/n\oh_K)^\times$.
\end{lem}
\begin{proof}
  This is straightforward given that $\oh_K=\ZZ+\ZZ \omega_K$ and $\oh_m=\ZZ+\ZZ m\omega_K$.

  Let $\varphi$ denote the ring homomorphism $\oh_m/n\oh_m\to\oh_K/n\oh_K$.
  \begin{itemize}
    \item For the kernel, suppose $x+ym\omega_K\in n\oh_K$, that is $x+ym\omega_K=n(a+b\omega_K)$, then
      \begin{equation*}
        x-na=(nb-ym)\omega_K=0,
      \end{equation*}
      so that $x=na$ and $ym=nb$.
      Since $\gcd(m,n)=1$ we must have $n\mid y$, so $x+ym\omega_K\in n\oh_m$.
    \item For surjectivity, consider an arbitrary element $a+b\omega_K+n\oh_K\in \oh_K/n\oh_K$.
      Since $\gcd(m,n)=1$, there exist $u,v\in\ZZ$ such that $um-vn=b$.
      Then
      \begin{equation*}
        \varphi(a+um\omega_K+n\oh_0)=a+um\omega_K+n\oh_K=a+(b+vn)\omega_K+n\oh_K=a+b\omega_K+n\oh_K.\qedhere
      \end{equation*}
  \end{itemize}
\end{proof}

\section{Group structure of $(\oh_K/\ell^d\oh_K)^\times$}

Fix a prime $\ell$.

The group structure of $(\ZZ/\ell^d\ZZ)^\times$ is a classical result going back to Euler.
We describe it using explicit generators and isomorphisms:
\begin{align*}
  (\ZZ/\ell^d\ZZ)^\times & =\langle g\rangle \to \ZZ/\ell^{d-1}\ZZ\times \ZZ/(\ell-1)\ZZ \text{ given by } g\mapsto (1,1) \text{ if }\ell\geq 3,                           \\
  (\ZZ/2\ZZ)^\times      & =\langle 1\rangle \to 0 \text{ given by } 1\mapsto 0,                                                                                           \\
  (\ZZ/2^2\ZZ)^\times    & =\langle -1\rangle \to \ZZ/2\ZZ \text{ given by }-1\mapsto 1,                                                                                   \\
  (\ZZ/2^d\ZZ)^\times    & =\langle 5\rangle\times\langle -1\rangle \to \ZZ/2^{d-2}\ZZ\times\ZZ/2\ZZ \text{ given by } 5\mapsto (1,0), -1\mapsto (0,1) \text{ if }d\geq 3.
\end{align*}
Throughout the paper, we fix these isomorphisms and use them to identify the multiplicative groups appearing on the left with the additive groups appearing on the right.
We also fix, for each $\ell^d$ with $\ell$ odd prime and $d\geq 1$, a primitive root $g$.

Now we consider a quadratic imaginary field $K$ with ring of integers $\oh_K$ and are interested in the group structure of $(\oh_K/\ell^d\oh_K)^\times$.
The group of units $(\oh_K/\LL^d)^\times$ for a prime ideal $\LL$ is determined in \cite{sittinger} (similar results appear in \cite[Chapter 6]{kohler}, but \cite{sittinger} gives generators for the cyclic factors).
However, in several cases a different choice of generators is more natural for our application, so we describe our use and modifications of \cite{sittinger} here.

\begin{lem}\label{inject_Z_to_OK}
  Fix $d \geq 1$, and let $\LL$ be a prime ideal of $\oh_K$ above the prime number $\ell$.
  The map $\ZZ\hookrightarrow\oh_K$ induces an injective ring homomorphism
  \begin{equation*}
    \ZZ/\ell^{\lceil d/e\rceil}\ZZ\into \oh_K/\LL^d,
  \end{equation*}
  where $e$ is the ramification index of $\ell$ in $\oh_K$.

  If $\ell$ is split then this is a ring isomorphism.
\end{lem}
\begin{proof}
  Let $\varphi$ denote the composition of the inclusion $\ZZ\to\oh_K$ with the quotient map $\oh_K\to\oh_K/\LL^d$.
  We have
  \begin{equation*}
    \ker \varphi = \LL^d\cap\ZZ = \ell^{\lceil d/e\rceil}\ZZ,
  \end{equation*}
  hence quotienting by the kernel gives the desired injective ring homomorphism.

  If $\ell$ is split we have $e=f=1$ and $N(\LL^d)=N(\LL)^d=\ell^d$ so the two rings have the same cardinality, hence the map is an isomorphism.
\end{proof}

Taking units, we deduce
\begin{cor}\label{inject_Z_to_OK_gps}
  We have an injective group homomorphism
  \begin{equation*}
    \iota\st (\ZZ/\ell^{\lceil d/e\rceil}\ZZ)^\times \into (\oh_K/\LL^d)^\times\qquad\text{mapping }1\mapsto 1.
  \end{equation*}
  If $\ell$ is split then this is a group isomorphism.
\end{cor}

\begin{rem}
  \cref{inject_Z_to_OK_gps} is a simple generalisation of \cite[Lemma 4]{sittinger}.
  The split case gives \cite[Theorem 5]{sittinger}.
\end{rem}

The purpose of this section is to elucidate the following short exact sequence:
\begin{equation}\label{eq:lambda_d}
  1 \to (\ZZ/\ell^d\ZZ)^\times \to (\oh_K/\ell^d\oh_K)^\times \to \lambda_d \to 1.
\end{equation}

\begin{thm}\label{lambda-classification}
  Let $d\geq 1$.
  In each of the following cases, the given conditions on $d$, $\ell$, and $\oh_K$ imply the stated group structure for $\lambda_d$.

  \begin{center}
    \begin{tabular}{llllll}
      \toprule
          & $d$      & $\ell$   & $\ell$ in $\oh_K$ & extra assumptions      & $\lambda_d$                              \\ \midrule
      (a) & $\geq 1$ & $\geq 3$ & split             & none                   & $\ZZ/\ell^{d-1}\ZZ\times\ZZ/(\ell-1)\ZZ$ \\
          &          &          & inert             & none                   & $\ZZ/\ell^{d-1}\ZZ\times\ZZ/(\ell+1)\ZZ$ \\ \midrule
      (b) & $\geq 1$ & $=2$     & ramified          & $D_K\equiv 8\mod{16}$  & $\ZZ/\ell^d\ZZ$                          \\
          &          & $=3$     &                   & $D_K\equiv 3\mod{9}$                                              \\
          &          & $\geq 5$ &                   & none                                                              \\ \midrule
      (c) & $\geq 1$ & $=2$     & ramified          & $D_K\equiv 12\mod{16}$ & $\ZZ/\ell^{d-1}\ZZ\times\ZZ/\ell\ZZ$     \\
          &          & $=3$     &                   & $D_K\equiv 6\mod{9}$                                              \\ \midrule
      (d) & $\geq 2$ & $=2$     & split or inert    & none                   & $\ZZ/2^{d-2}\ZZ\times\ZZ/2\ZZ$           \\
      \bottomrule
    \end{tabular}
  \end{center}

  The remaining special cases (in (d)) are:
  \begin{itemize}
    \item For $\ell=2$ split and $d<2$ we have $\lambda_1\cong 0$.
    \item For $\ell=2$ inert and $d<2$ we have $\lambda_1\cong\ZZ/3\ZZ$.
  \end{itemize}
\end{thm}

\begin{rem}
  Note that the congruence classes for $D_K$ listed for $\ell=2,3$ are the only possible ones given that we assume $\ell$ to be ramified in $\oh_K$: $2$ is ramified in $\oh_K$ if and only if $D_K\equiv 0\mod{4}$, so $D_K=-4y$ with $y$ positive, squarefree, and $\equiv 1,2\mod{4}$.
  Similarly, $3$ is ramified in $\oh_K$ if and only if $D_K\equiv 0\mod{3}$.
\end{rem}

\begin{proof}[Proof of \cref{lambda-classification}]
  We proceed according to the type of ramification of $\ell$ in $\oh_K$.
  \begin{enumerate}
    \item \textbf{split}:
      Write $\ell\oh_K=\LL\LLbar$.
      \cref{inject_Z_to_OK_gps} allows us to identify $(\oh_K/\LL^d)^\times$ with $(\ZZ/\ell^d\ZZ)^\times$ (and similarly for $\LLbar^d$).
      By the Chinese Remainder Theorem the exact sequence \eqref{eq:lambda_d} is
      \begin{equation*}
        1\to (\ZZ/\ell^d\ZZ)^\times \xrightarrow{\;\;\iota\;\;} (\ZZ/\ell^d\ZZ)^\times \times (\ZZ/\ell^d\ZZ)^\times\to \lambda_d\to 1,
      \end{equation*}
      where the embedding $\iota$ is diagonal (due to the non-trivial Galois automorphism of $K$ that interchanges $\LL$ and $\LLbar$).

      If $\ell\neq 2$ (or $\ell=2$ and $d=1$) we have $(\ZZ/\ell^d\ZZ)^\times=\langle g\rangle$ with $\iota(g)=(g,g)$ so $\lambda_d$ is generated by the image of $(g,1)$.
      We pass to the additive versions of these groups:
      \begin{equation*}
        \begin{tikzcd}[column sep=2em,row sep=1.5em,trim left=-7.5cm]
          (\ZZ/\ell^d\ZZ)^\times \arrow[r, "\iota"] \arrow[d, "\sim"] & (\ZZ/\ell^d\ZZ)^\times \times (\ZZ/\ell^d\ZZ)^\times \arrow[r] \arrow[d,"\sim"] & \lambda_d\\
          \ZZ/\ell^{d-1}\ZZ \times \ZZ/(\ell-1)\ZZ \arrow[r, "j"] & (\ZZ/\ell^{d-1}\ZZ\times\ZZ/(\ell-1)\ZZ)\times (\ZZ/\ell^{d-1}\ZZ\times\ZZ/(\ell-1)\ZZ),
        \end{tikzcd}
      \end{equation*}
      where $j(1)=(1,1)$, so the cokernel of $j$ is generated by the image of $(1,0)$.
      We can therefore identify $\lambda_d$ with $\coker j\cong\ZZ/\ell^{d-1}\ZZ\times\ZZ/(\ell-1)\ZZ$ via the isomorphism that takes $(g,1)$ to $(1,0)$.

      For $\ell=2$ and $d\geq 2$ we have $(\ZZ/2^d\ZZ)^\times\cong \langle 5\rangle \times \langle -1\rangle$ with $\iota(5)=(5,5)$ and $\iota(-1)=(-1,-1)$, so $\lambda_d$ is generated by the images of $(5,1)$ and $(1,-1)$.
      We have
      \begin{equation*}
        \begin{tikzcd}[column sep=2em,row sep=1.5em,trim left=-7.5cm]
          (\ZZ/2^d\ZZ)^\times \arrow[r, "\iota"] \arrow[d, "\sim"] & (\ZZ/2^d\ZZ)^\times \times (\ZZ/2^d\ZZ)^\times \arrow[r] \arrow[d,"\sim"] & \lambda_d\\
          \ZZ/2^{d-2}\ZZ \times \ZZ/2\ZZ \arrow[r, "j"] & (\ZZ/2^{d-2}\ZZ\times\ZZ/2\ZZ)\times (\ZZ/2^{d-2}\ZZ\times\ZZ/2\ZZ)
        \end{tikzcd}
      \end{equation*}
      where the generators $(1,0)$ and $(0,1)$ of $\ZZ/2^{d-2}\ZZ\times\ZZ/2\ZZ$ are mapped as follows:
      \begin{equation*}
        j(1,0)=\big((1,0),(1,0)\big)\qquad\text{and}\qquad j(0,1)=\big((0,1),(0,1)\big).
      \end{equation*}
      Therefore the cokernel of $j$ is generated by the images of $\big((1,0),(0,0)\big)$ (of order $2^{d-2}$) and $\big((0,1),(0,0)\big)$ (of order $2$).
      We can identify $\lambda_d$ with $\coker j\cong\ZZ/2^{d-2}\ZZ\times\ZZ/2\ZZ$ via the isomorphism that takes $(5,1)$ to $\big((1,0),(0,0)\big)$ and $(1,-1)$ to $\big((0,1),(0,0)\big)$.

    \item \textbf{inert}:
      We have $e=1$ so \cref{inject_Z_to_OK_gps} gives our embedding
      \begin{equation*}
        \iota\st (\ZZ/\ell^{d}\ZZ)^\times \into (\oh_K/\ell^d\oh_K)^\times.
      \end{equation*}
      If $\ell\neq 2$, we modify the argument in \cite[Theorem 6]{sittinger} as follows.

      Let $\beta\in\oh_K$ be such that its image in $(\oh_K/\ell\oh_K)^\times\cong\FF_{\ell^2}^\times$ is a generator of this cyclic group.
      Then the image of $\beta^{\ell^{d-1}}$ in $(\oh_K/\ell^d\oh_K)^\times$ has order $\ell^2-1$.
      Let $\zeta$ be the image of $\beta^{\ell^{d-1}(\ell-1)}$, then $\zeta$ is an element of order $\ell+1$.

      By the arguments in \cite[Lemmas 2 and 3]{sittinger}, the element $1+\ell\zeta$ has order $\ell^{d-1}$.

      Finally, the element $g$ has order $\ell^{d-1}(\ell-1)$.

      We claim that $(\oh_K/\ell^d\oh_K)^\times=\langle g\rangle \times \langle 1+\ell\zeta\rangle \times \langle \zeta\rangle$.

      Since $\iota$ maps isomorphically onto the first factor $\langle g\rangle$, the cokernel $\lambda_d$ is generated by the images of $1+\ell\zeta$ and of $\zeta$.
      We can identify $\lambda_d$ with $\ZZ/\ell^{d-1}\ZZ\times\ZZ/(\ell+1)\ZZ$ via the isomorphism that takes $1+\ell\zeta$ to $(1,0)$ and $\zeta$ to $(0,1)$.

      For the case $\ell=2$ we modify \cite[Theorem 7]{sittinger} along similar lines.
      The upshot is that, if $d\geq 2$:
      \begin{equation*}
        (\oh_K/2^d\oh_K)^\times = \langle 1+2\zeta\rangle \times \langle -1\rangle \times \langle 1+4\zeta \rangle \times \langle \zeta\rangle,
      \end{equation*}
      where $\zeta$ is an element of order $3$, $(1+2\zeta)^2=5$ so $1+2\zeta$ has order $2^{d-1}$, and $1+4\zeta$ has order $2^{d-2}$.

      Moreover, $\iota$ maps $5$ to $(1+2\zeta)^2$ and $-1$ to $-1$, so that the cokernel $\lambda_d$ is generated by the images of $1+2\zeta$ (of order $2$) and of $1+4\zeta$ (of order $2^{d-2}$).
      We can identify $\lambda_d$ with $\ZZ/2\ZZ\times\ZZ/2^{d-2}\ZZ$ via the isomorphism that takes $1+2\zeta$ to $(1,0)$ and $1+4\zeta$ to $(0,1)$.

      The special case $d=1$ is easily treated: $(\oh_K/2\oh_K)^\times\cong\FF_4^\times\cong\ZZ/3\ZZ$ and $(\ZZ/2\ZZ)^\times\cong 0$, so $\lambda_1$ is generated by the image of a generator $\zeta$ of $(\oh_K/2\oh_K)^\times$, and can be identified with $\ZZ/3\ZZ$ by mapping $\zeta$ to $1$.
    \item \textbf{ramified}:
      This time $e=2$ and \cref{inject_Z_to_OK_gps} gives the embedding
      \begin{equation*}
        \iota\st (\ZZ/\ell^{d}\ZZ)^\times \into (\oh_K/\LL^{2d})^\times=(\oh_K/\ell^d\oh_K)^\times.
      \end{equation*}
      Suppose $\ell\geq 5$.
      We can use \cite[Theorem 8]{sittinger} as is:
      \begin{equation*}
        (\oh_K/\ell^d\oh_K)^\times = \langle g\rangle \times \langle 1+\sqrt{N}\rangle,
      \end{equation*}
      where $1+\sqrt{N}$ has order $\ell^d$.
      As $\iota$ maps isomorphically onto the first factor, the cokernel $\lambda_d$ is generated by the image of $1+\sqrt{N}$.
      We can identify $\lambda_d$ with $\ZZ/\ell^d\ZZ$ by mapping $1+\sqrt{N}$ to $1$.

      The same reasoning applies to the case $\ell=3$, where we can use \cite[Theorems 9 and 10]{sittinger} as stated, giving:
      \begin{itemize}
        \item if $D_K\equiv 3\mod{9}$, $\lambda_d$ is generated by the image of $1+\sqrt{N}$, an element of order $3^d$;
        \item if $D_K\equiv 6\mod{9}$, $\lambda_d$ is generated by the images of $1+3\sqrt{N}$ (of order $3^{d-1}$) and $\zeta$ (of order $3$).
      \end{itemize}

      It remains to deal with the case $\ell=2$.
      For $d\geq 3$, following \cite[Theorem 12]{sittinger}, we consider congruence classes
      \begin{itemize}
        \item $N\equiv 2\mod{4}$, so $D_K\equiv 8\mod{16}$.
          We use \cite[Theorem 12]{sittinger} (a) as is and get that $\lambda_d$ is generated by the image of $1+\sqrt{N}$, an element of order $2^d$.
        \item $N\equiv 7\mod{8}$, so $D_K\equiv 28\mod{32}$.
          Here \cite[Theorem 12]{sittinger} (b) says that
          \begin{equation*}
            (\oh_K/2^d\oh_K)^\times = \langle 5\rangle \times \langle \zeta\rangle \times \langle 1+2\sqrt{N}\rangle,
          \end{equation*}
          where $\zeta$ satisfies $\zeta^2=-1$ (see \cite[bottom of page 18]{sittinger}) and $1+2\sqrt{N}$ has order $2^{d-1}$.

          So $\iota$ maps $\langle 5\rangle$ isomorphically onto the first factor and maps $\langle -1\rangle$ to $\langle \zeta^2\rangle$ in the second factor.
          Therefore the cokernel $\lambda_d$ is generated by the images of $\zeta$ and of $1+2\sqrt{N}$.
          We can identify $\lambda_d$ with $\ZZ/2\ZZ\times\ZZ/2^{d-1}\ZZ$ by mapping $\zeta$ to $(1,0)$ and $1+2\sqrt{N}$ to $(0,1)$.
        \item $N\equiv 3\mod{8}$, so $D_K\equiv 12\mod{32}$.
          Here we note that a simple modification of the proof of \cite[Theorem 12]{sittinger} (c) gives
          \begin{equation*}
            (\oh_K/2^d\oh_K)^\times = \langle \zeta\rangle \times \langle -1\rangle \times \langle 1+2\sqrt{N}\rangle,
          \end{equation*}
          where $\zeta$ satisfies $\zeta^2=5$ and $1+2\sqrt{N}$ has order $2^{d-1}$.

          The embedding $\iota$ sends $\langle 5\rangle$ to $\langle \zeta^2\rangle$ in the first factor, and sends $\langle -1\rangle$ isomorphically to the second factor.
          Therefore the cokernel $\lambda_d$ is generated by the images of $\zeta$ and $1+2\sqrt{N}$.
          We can identify $\lambda_d$ with $\ZZ/2\ZZ\times\ZZ/2^{d-1}\ZZ$ by mapping $\zeta$ to $(1,0)$ and $1+2\sqrt{N}$ to $(0,1)$.
      \end{itemize}
      In hindsight, we can group the results for $D_K\equiv 28\mod{32}$ and $D_K\equiv 12\mod{32}$ into a single case $D_K\equiv 12\mod{16}$.

      The $d=1$ and $d=2$ cases follow easily from \cite[Theorem 11]{sittinger}.
      \qedhere
  \end{enumerate}
\end{proof}
\begin{rem}\label{lambda_1}
  We point out that in the case $d=1$, \cref{lambda-classification} gives
  \begin{equation*}
    \lambda_1 \cong \begin{cases}
      \ZZ/(\ell-1)\ZZ & \text{if $\ell$ is split,}    \\
      \ZZ/(\ell+1)\ZZ & \text{if $\ell$ is inert,}    \\
      \ZZ/\ell\ZZ     & \text{if $\ell$ is ramified.}
    \end{cases}
  \end{equation*}
\end{rem}

\section{Ranks}\label{sec:ranks}

Let $S$ be an abelian $\ell$-group and consider the multiplication-by-$\ell$ map on $S$:
\begin{equation}\label{eq:mult_by_ell}
  \begin{tikzcd}[column sep=2em, row sep=1.5em]
    0 \arrow[r]
    & S[\ell] \arrow[r]
    & S \arrow[r, "\times \ell"]
    & S \arrow[r]
    & S/\ell S \arrow[r]
    & 0.
  \end{tikzcd}
\end{equation}
The rank of $S$, denoted $\rk(S)$, is the dimension of $S[\ell]$ as an $\FF_\ell$-vector space.

Consider the decomposition of $S$ into cyclic factors:
\begin{equation*}
  S \cong \bigoplus_{i = 1}^t \ZZ/\ell^{a_i}\ZZ, \qquad a_1\geq\cdots\geq a_t \geq 1.
\end{equation*}
Each summand contributes one copy of $\ZZ/\ell\ZZ$ to the subgroup $S[\ell]$, so that $S[\ell]\cong (\ZZ/\ell\ZZ)^t$.
The rank of $S$ is thus $t$, the number of summands.
Similarly, we have an isomorphism $S/\ell S \cong (\ZZ/\ell\ZZ)^t$.
So also $\rk(S) = \dim_{\FF_\ell} S/\ell S$.

If $0\to A\to B\to C\to 0$ is a short exact sequence of finite abelian groups, then applying the Snake Lemma to
\begin{equation*}
  \begin{tikzcd}
    0 \arrow[r]
    & A \arrow[r] \arrow[d, "\times \ell"]
    & B \arrow[r] \arrow[d, "\times \ell"]
    & C \arrow[r] \arrow[d, "\times \ell"]
    & 0 \\
    0 \arrow[r]
    & A \arrow[r]
    & B \arrow[r]
    & C \arrow[r]
    & 0.
  \end{tikzcd}
\end{equation*}
gives a long exact sequence
\begin{equation}\label{eq:les}
  0\to A[\ell]\to B[\ell]\to C[\ell]\xrightarrow{\;\;\delta\;\;} A/\ell A\to B/\ell B\to C/\ell C \to 0.
\end{equation}
We can therefore express the rank of $B$ as
\begin{equation*}
  \rk(B) = \rk(A) + \rk(C) - \dim_{\FF_\ell} \im\delta.
\end{equation*}
Note also that $\dim\im\delta\leq \rk(A)$.

In our setting, let $S_d$ be the Sylow $\ell$-subgroup of $\Cl(\oh_d)$.
We have the following commutative diagram:
\begin{equation*}
  \begin{tikzcd}
    0 \arrow[r]
    & \kappa_{d+1} \arrow[r] \arrow[d, "\pi'"]
    & S_{d+1} \arrow[r] \arrow[d, "\pi"]
    & S_0 \arrow[r] \arrow[d, equal]
    & 0 \\
    0 \arrow[r]
    & \kappa_d \arrow[r]
    & S_d \arrow[r]
    & S_0 \arrow[r]
    & 0,
  \end{tikzcd}
\end{equation*}
where $\pi$ is the natural surjection $S_{d+1}\to S_d$ and $\pi'$ is its restriction to $\kappa_{d+1}$.
Note that $\pi'$ is also surjective (apply the Snake Lemma to the above diagram).

Taking $\ell$-torsion in each of the two rows gives two long exact sequences; by the naturality of the Snake Lemma (see \cite[paragraph following Lemma 5.1]{hilton-stammbach}), these form a commutative diagram
\begin{equation*}
  \begin{tikzcd}[column sep=1.5em]
    0 \arrow[r]
    & \kappa_{d+1}[\ell] \arrow[r] \arrow[d, "{\pi'[\ell]}"]
    & S_{d+1}[\ell] \arrow[r] \arrow[d, "{\pi[\ell]}"]
    & S_0[\ell] \arrow[r, "\delta_{d+1}"] \arrow[d, equal]
    & \kappa_{d+1}/\ell \kappa_{d+1} \arrow[r] \arrow[d, "\bar{\pi}'"]
    & S_{d+1}/\ell S_{d+1} \arrow[r] \arrow[d, "\bar{\pi}"]
    & S_0/\ell S_0 \arrow[r] \arrow[d, equal]
    & 0 \\
    0 \arrow[r]
    & \kappa_d[\ell] \arrow[r]
    & S_d[\ell] \arrow[r]
    & S_0[\ell] \arrow[r, "\delta_d"]
    & \kappa_d/\ell \kappa_d \arrow[r]
    & S_d/\ell S_d \arrow[r]
    & S_0/\ell S_0 \arrow[r]
    & 0,
  \end{tikzcd}
\end{equation*}
In particular, the commutative square in the centre shows that $\delta_d = \bar{\pi}'\circ\delta_{d+1}$.
Note that $\bar{\pi}'$ is surjective since $\pi'$ is (and tensoring with $\ZZ/\ell\ZZ$ is a right exact functor).

We conclude that
\begin{lem}\label{lem:equal_rank}
  \begin{equation*}
    \rk(S_{d+1})-\rk(S_d) = \dim\ker\bar{\pi}'-\dim\big(\ker\bar{\pi}'\cap \im\delta_{d+1}\big),
    \text{ where $\dim\ker\bar{\pi}'=\rk(\kappa_{d+1})-\rk(\kappa_d)$}.
  \end{equation*}
\end{lem}

\begin{thm} \label{class-group-ranks}
  Fix $\ell\geq 3$.
  \begin{enumerate}
    \item Suppose $\ell$ is split or inert.
      Then
      \begin{align*}
        \rk(S_1) & = \rk(S_0)                         \\
        \rk(S_2) & = \rk(S_0)\text{ or }\rk(S_0)+1    \\
        \rk(S_d) & = \rk(S_2)\text{ for all }d\geq 2.
      \end{align*}
    \item Suppose $\ell$ is ramified and:
      \begin{equation*}
        \text{$\ell\geq 5$ or ($\ell=3$ and $D_K\equiv 3\mod{9}$).}
      \end{equation*}
      Then
      \begin{align*}
        \rk(S_1) & = \rk(S_0)\text{ or }\rk(S_0)+1    \\
        \rk(S_d) & = \rk(S_1)\text{ for all }d\geq 1.
      \end{align*}
    \item Suppose $\ell=3$ is ramified and $D_K\equiv 6\mod{9}$.
      Then
      \begin{align*}
        \rk(S_1) & = \rk(S_0)\text{ or }\rk(S_0)+1    \\
        \rk(S_2) & = \rk(S_1)\text{ or }\rk(S_1)+1    \\
        \rk(S_d) & = \rk(S_2)\text{ for all }d\geq 2.
      \end{align*}
  \end{enumerate}
\end{thm}
\begin{proof}
  All the parts are treated in a similar way: we determine the rank of $\kappa_d$ from the relevant part of \cref{kappa-classification}, then combine it with \cref{lem:equal_rank} to relate the ranks of $S_{d+1}$ and $S_d$.
  We give more details in part (a), and describe the main differences in the other parts.
  \begin{enumerate}
    \item
      By \ref{kappa-class-split-inert-general} we have $\kappa_d\cong\ZZ/\ell^{d-1}\ZZ$, then using \cref{lem:equal_rank}:
      \begin{itemize}
        \item $\kappa_1[\ell]=0$, so trivially $\rk(S_1)=\rk(S_0)$.
        \item For $d\geq 2$, $\rk(\kappa_d)=1$ so we get $\rk(S_{d+1})=\rk(S_d)$.
        \item For $d=1$, $\dim\ker\bar{\pi}'=1$ so
          \begin{equation*}
            \rk(S_2) - \rk(S_1) = 1 - \dim\big(\ker\bar{\pi}'\cap\im\delta_2\big) \in \{0,1\}.
          \end{equation*}
      \end{itemize}
    \item Use $\rk(\kappa_d)=1$ for all $d\geq 1$.
    \item Use $\rk(\kappa_1)=1$ and $\rk(\kappa_d)=2$ for all $d\geq 2$.\qedhere
  \end{enumerate}
\end{proof}

In the case $\ell=2$, we get more precise information directly from the calculation of $\card{S_d[2]}$ from \cref{ses-card-holds}:
\begin{prop}
  Let $\ell=2$.
  \begin{enumerate}
    \item If $2$ is split or inert, then
      \begin{align*}
        \rk(S_1) & = \rk(S_0)                           \\
        \rk(S_2) & = \rk(S_0)+1                         \\
        \rk(S_d) & = \rk(S_0)+2\text{ for all }d\geq 3.
      \end{align*}
    \item If $2$ is ramified and $D_K\equiv 8\mod{16}$, then
      \begin{align*}
        \rk(S_d) & = \rk(S_0)+1\text{ for all }d\geq 1.
      \end{align*}
    \item If $2$ is ramified and $D_K\equiv 12\mod{16}$, then
      \begin{align*}
        \rk(S_1) & = \rk(S_0)                           \\
        \rk(S_d) & = \rk(S_0)+1\text{ for all }d\geq 2.
      \end{align*}
  \end{enumerate}
\end{prop}

\bibliographystyle{plain}
\bibliography{refs}

\end{document}

%% file: header.tex
\usepackage[l2tabu,orthodox]{nag}

\usepackage[a4paper, hmargin=2.2cm, vmargin=2.0cm, marginparwidth=2cm, marginparsep=0.2cm]{geometry}

\usepackage[shortlabels]{enumitem}
\usepackage{numprint}
\usepackage[short]{datetime}

\usepackage{booktabs}
\usepackage{mathtools}
\usepackage{amsthm}
\usepackage{MnSymbol}
\usepackage[protrusion=true,verbose=true,final=true]{microtype}

\usepackage[hypertexnames=false]{hyperref}
\hypersetup{
  final=true,
  colorlinks=true,
  linkcolor=blue,
  urlcolor=blue,
  citecolor=magenta
}

\usepackage{tikz}
\usetikzlibrary{spy}
\usetikzlibrary{matrix,arrows,positioning}
\usetikzlibrary{decorations.markings}
\usetikzlibrary{decorations.fractals}
\usetikzlibrary{decorations.pathmorphing}
\usetikzlibrary{shapes.geometric}
\usetikzlibrary{calc}
\usetikzlibrary{decorations.pathreplacing}

\usetikzlibrary{graphs,arrows.meta}
\definecolor{fashionfuchsia}{rgb}{0.96, 0.0, 0.63}
\tikzset{
  vertex/.style={circle,draw=black,fill=fashionfuchsia,inner sep=2pt},
  edge/.style={thick},
}

\usepackage{tikz-cd}

\newcommand{\ZZ}{\mathbf{Z}}
\newcommand{\QQ}{\mathbf{Q}}
\newcommand{\FF}{\mathbf{F}}

\newcommand{\NN}{\mathbf{N}}

\newcommand{\LL}{\mathcal{L}}
\newcommand{\LLbar}{\widebar{\mathcal{L}}}
\newcommand{\fp}{\FF_p}
\newcommand{\fpk}{\FF_{p^k}}
\newcommand{\fpbar}{\overline{\FF}_p}
\newcommand{\Aut}{\operatorname{Aut}}
\newcommand{\Ext}{\operatorname{Ext}}
\newcommand{\Gal}{\operatorname{Gal}}

\newcommand{\End}{\operatorname{End}}

\newcommand{\im}{\operatorname{im}}
\newcommand{\id}{\operatorname{id}}

\newcommand{\rk}{\operatorname{rk}}

\newcommand{\Tr}{\operatorname{Tr}}

\newcommand{\longto}{\longrightarrow}

\newcommand{\Cl}{\operatorname{Cl}}
\newcommand{\ord}{\operatorname{ord}}

\newcommand{\can}{\operatorname{can}}
\newcommand{\coker}{\operatorname{coker}}
\newcommand{\Hom}{\operatorname{Hom}}

\newcommand{\oh}{\mathcal{O}}

\newcommand{\into}{\hookrightarrow}
\newcommand{\pp}{\mathfrak{p}}

\newcommand{\qq}{\mathfrak{q}}

\newcommand{\st}{\,\colon\;}

\newcommand{\Etilde}{\widetilde{E}}
\newcommand{\alphatilde}{\widetilde{\alpha}}

\newcommand{\gpk}{\mathcal{G}_\ell(\fpk)}

\renewcommand{\geq}{\geqslant}
\renewcommand{\ge}{\geqslant}
\renewcommand{\leq}{\leqslant}

\renewcommand{\to}{\longto}
\renewcommand{\mapsto}{\longmapsto}

\newcommand{\hash}{\textnormal{\texttt{\#}}}
\newcommand{\card}[1]{\hash #1}

\usepackage{thmtools}
\usepackage[noabbrev,capitalize,nameinlink]{cleveref}

\newtheorem{thm}{Theorem}[section]
\crefname{thm}{Theorem}{Theorems}

\crefname{conj}{Conjecture}{Conjectures}
\newtheorem{lem}[thm]{Lemma}
\crefname{lem}{Lemma}{Lemmas}
\newtheorem{cor}[thm]{Corollary}
\crefname{cor}{Corollary}{Corollaries}
\newtheorem{prop}[thm]{Proposition}
\crefname{prop}{Proposition}{Propositions}
\theoremstyle{definition}
\newtheorem{definition}[thm]{Definition}

\crefname{exm}{Example}{Examples}
\newtheorem{rem}[thm]{Remark}
\crefname{rem}{Remark}{Remarks}

\crefname{exe}{Exercise}{Exercises}

\DeclareFontFamily{U}{mathx}{\hyphenchar\font45}
\DeclareFontShape{U}{mathx}{m}{n}{<-> mathx10}{}
\DeclareSymbolFont{mathx}{U}{mathx}{m}{n}
\DeclareMathAccent{\widebar}{0}{mathx}{"73}

\DeclarePairedDelimiter{\abs}{\lvert}{\rvert}

\newcommand{\define}[1]{\textit{\color{magenta}{#1}}}
\newcommand{\explosive}{$k$-explosive}
\newcommand{\cohlen}{Cohen--Lenstra}

\newcounter{tabenum}
\renewcommand{\thetabenum}{(\alph{tabenum})}
\newcommand{\tabitem}[1]{%
  \makebox[0pt][l]{\raisebox{2.5ex}[0pt][0pt]{\refstepcounter{tabenum}\label{#1}}}%
  \thetabenum%
}
\makeatletter
\newcommand{\settabprefix}[1]{%
  \renewcommand{\p@tabenum}{#1}%
}
\makeatother

\usepackage{subfig}